\documentclass[11pt,twoside]{article}

\usepackage{fullpage}

\usepackage{epsf}
\usepackage{fancyheadings}
\usepackage{graphics}
\usepackage{graphicx}
\usepackage{psfrag}

\usepackage{color}

\usepackage{amsthm}
\usepackage{amsfonts}
\usepackage{amsmath}
\usepackage{amssymb}

\usepackage[linesnumbered, ruled, vlined]{algorithm2e}

\theoremstyle{plain}

\newtheorem{theorem}{Theorem}
\newtheorem{proposition}{Proposition}
\newtheorem{lemma}{Lemma}

\newtheorem{definition}{Definition}
\newtheorem{remark}{Remark}
\newlength{\widebarargwidth}
\newlength{\widebarargheight}
\newlength{\widebarargdepth}
\DeclareRobustCommand{\widebar}[1]{%
  \settowidth{\widebarargwidth}{\ensuremath{#1}}%
  \settoheight{\widebarargheight}{\ensuremath{#1}}%
  \settodepth{\widebarargdepth}{\ensuremath{#1}}%
  \addtolength{\widebarargwidth}{-0.3\widebarargheight}%
  \addtolength{\widebarargwidth}{-0.3\widebarargdepth}%
  \makebox[0pt][l]{\hspace{0.3\widebarargheight}%
    \hspace{0.3\widebarargdepth}%
    \addtolength{\widebarargheight}{0.3ex}%
    \rule[\widebarargheight]{0.95\widebarargwidth}{0.1ex}}%
  {#1}}

\makeatletter
\long\def\@makecaption#1#2{
        \vskip 0.8ex
        \setbox\@tempboxa\hbox{\small {\bf #1:} #2}
        \parindent 1.5em  
        \dimen0=\hsize
        \advance\dimen0 by -3em
        \ifdim \wd\@tempboxa >\dimen0
                \hbox to \hsize{
                        \parindent 0em
                        \hfil 
                        \parbox{\dimen0}{\def\baselinestretch{0.96}\small
                                {\bf #1.} #2
                                } 
                        \hfil}
        \else \hbox to \hsize{\hfil \box\@tempboxa \hfil}
        \fi
        }
\makeatother

\long\def\comment#1{}

\newcommand{\dH}{\mathsf{d}_{\mathsf{H}} }

\newcommand{\1}{\ensuremath{{\sf (i)}}}
\newcommand{\2}{\ensuremath{{\sf (ii)}}}
\newcommand{\3}{\ensuremath{{\sf (iii)}}}
\newcommand{\4}{\ensuremath{{\sf (iv)}}}
\newcommand{\5}{\ensuremath{{\sf (v)}}}

\DeclareMathOperator{\cov}{cov}
\DeclareMathOperator{\trace}{trace}

\newcommand{\NORMAL}{\ensuremath{\mathcal{N}}}

\newcommand{\Fspace}{\ensuremath{\mathcal{F}}}

\newcommand{\xhat}{\ensuremath{\widehat{x}}}

\newcommand{\real}{\ensuremath{\mathbb{R}}}

\newcommand{\Pihat}{\ensuremath{\widehat{\Pi}}}
\newcommand{\Pihatml}{\ensuremath{\widehat{\Pi}_{{\sf ML}}}}

\newcommand{\Projorthpi}{\ensuremath{P_{\Pi}^{\perp}}}
\newcommand{\Projorthpistar}{\ensuremath{P_{\Pi^*}^{\perp}}}

\newcommand{\hballpistar}{\ensuremath{\mathbb{B}_{\sf H}}}
\newcommand{\SNR}{\ensuremath{\frac{\|x^*\|_2^2}{\sigma^2}}}
\newcommand{\snr}{\ensuremath{{\sf snr}}}

\newcommand{\EE}{\ensuremath{\mathbb{E}}}
\newcommand{\Rpihat}{\ensuremath{\Pr\{ \Pihat \neq \Pi^* \} }}

\newcommand{\Rpihatml}{\ensuremath{\Pr\{ \Pihatml \neq \Pi^* \}}}
\newcommand{\BigO}{\ensuremath{\mathcal{O}}}

\begin{document}


\begin{center}

{\bf{\LARGE{Linear Regression with an Unknown Permutation: Statistical and Computational Limits}}}

\vspace*{.2in}

{\large{
\begin{tabular}{ccc}
Ashwin Pananjady$^\dagger$ &  Martin J. Wainwright$^{\dagger,\star}$ & Thomas A. Courtade$^\dagger$ \\
\end{tabular}
}}

\vspace*{.2in}

\begin{tabular}{c}
Department of Electrical Engineering and Computer Sciences$^\dagger$ \\
Department of Statistics$^\star$ \\
UC Berkeley
\end{tabular}

\vspace*{.2in}

\today

\vspace*{.2in}

\begin{abstract}
Consider a noisy linear observation model with an unknown permutation, based on observing $y = \Pi^* A x^* + w$, where $x^* \in \real^d$ is an unknown vector, $\Pi^*$ is an unknown $n \times n$ permutation matrix, and $w \in \real^n$ is additive Gaussian noise. We analyze the problem of permutation recovery in a random design setting in which the entries of the matrix $A$ are drawn i.i.d. from a standard Gaussian distribution, and establish sharp conditions on the SNR, sample size $n$, and dimension $d$ under which $\Pi^*$ is exactly and approximately recoverable. 
On the computational front, we show that the maximum likelihood estimate of $\Pi^*$ is NP-hard to compute, while also providing a polynomial time algorithm when $d =1$.
\end{abstract}

\end{center}


\section{Introduction}
\label{SecIntroduction}
Recovery of a vector based on noisy linear measurements is the classical problem of linear regression, and is arguably the most basic problem in statistical inference. A variant, the ``errors-in-variables" model \cite{little2014statistical}, allows for errors in the measurement matrix, but mainly in the form of additive or multiplicative noise \cite{loh2012corrupted}. In this paper, we study a form of errors-in-variables in which the measurement matrix is perturbed by an unknown permutation of its rows.

More concretely, we study an observation model of the form 
\begin{align}
y = \Pi^* A x^* + w, \label{obsmodel}
\end{align}
where $x^* \in \real^d$ is an unknown vector, $A \in \real^{n\times d}$ is a measurement (or design) matrix, $\Pi^*$ is an unknown $n\times n$ permutation matrix, and $w \in \real^n$ is observation noise. We refer to the setting where $w= 0$ as the \emph{noiseless case}. 
As with linear regression, there are two settings of interest, corresponding to whether the design matrix is \textbf{(i)} deterministic (the fixed design case), or \textbf{(ii)}~random (the random design case).

There are also two complementary problems of interest -- recovery of the unknown $\Pi^*$, and recovery of the unknown $x^*$. In this paper, we focus on the former problem; the latter problem is also known as unlabelled sensing \cite{unl}.

The observation model~\eqref{obsmodel} is frequently encountered in scenarios where there is uncertainty in the order in which measurements are taken. An illustrative example is that of sampling in the presence of jitter \cite{balakrishnan1962problem}, in which the uncertainty about the instants at which measurements are taken results in an unknown permutation of the measurements. A similar synchronization issue occurs in timing and molecular channels \cite{rose1}. Here, identical molecular tokens are received at the receptor at different times, and their signatures are indistinguishable. The vectors of transmitted and received times correspond to the signal and the observations, repectively, where the latter is some permuted version of the former with additive noise.

Another such scenario arises in multi-target tracking problems \cite{poore2006some}. For example, SLAM tracking \cite{thrun2008simultaneous} is a classical problem in robotics where the environment in which measurements are made is unknown, and part of the problem is to infer relative permutations between measurements. Archaeological measurements \cite{robinson1951method} also suffer from an inherent lack of ordering, which makes inference of chronology hard. Another compelling example of such an observation model  is in data anonymization, in which the order, or ``labels", of measurements are intentionally deleted to preserve privacy. The inverse problem of data de-anonymization \cite{narayanan2008robust} is to infer these labels from the observations.

Also, in large sensor networks, it is often the case that the number of bits of information that each sensor records and transmits to the server is exceeded by the number of bits it transmits in order to identify itself to the server \cite{keller2009identity}. In applications where sensor measurements are linear, model~\eqref{obsmodel} corresponds to the case where each sensor only sends its measurement but not its identity. The server is then tasked with recovering sensor identites, or equivalently, with determining the unknown permutation. 

The pose and correspondence estimation problem in image processing \cite{david2004softposit, marques2009subspace} is also related to the observation model~\eqref{obsmodel}. The capture of a 3D object by a 2D image can be modelled by an unknown linear transformation called the ``pose", and an unknown permutation representing the ``correspondence" between points in the two spaces. One of the central goals in image processing is to identify this correspondence information, which in this case is equivalent to permutation estimation in the linear model. An illustration of the problem is provided in Figure~\ref{fig:pose}.

\begin{figure}
    \begin{center}
    \includegraphics[scale=.35]{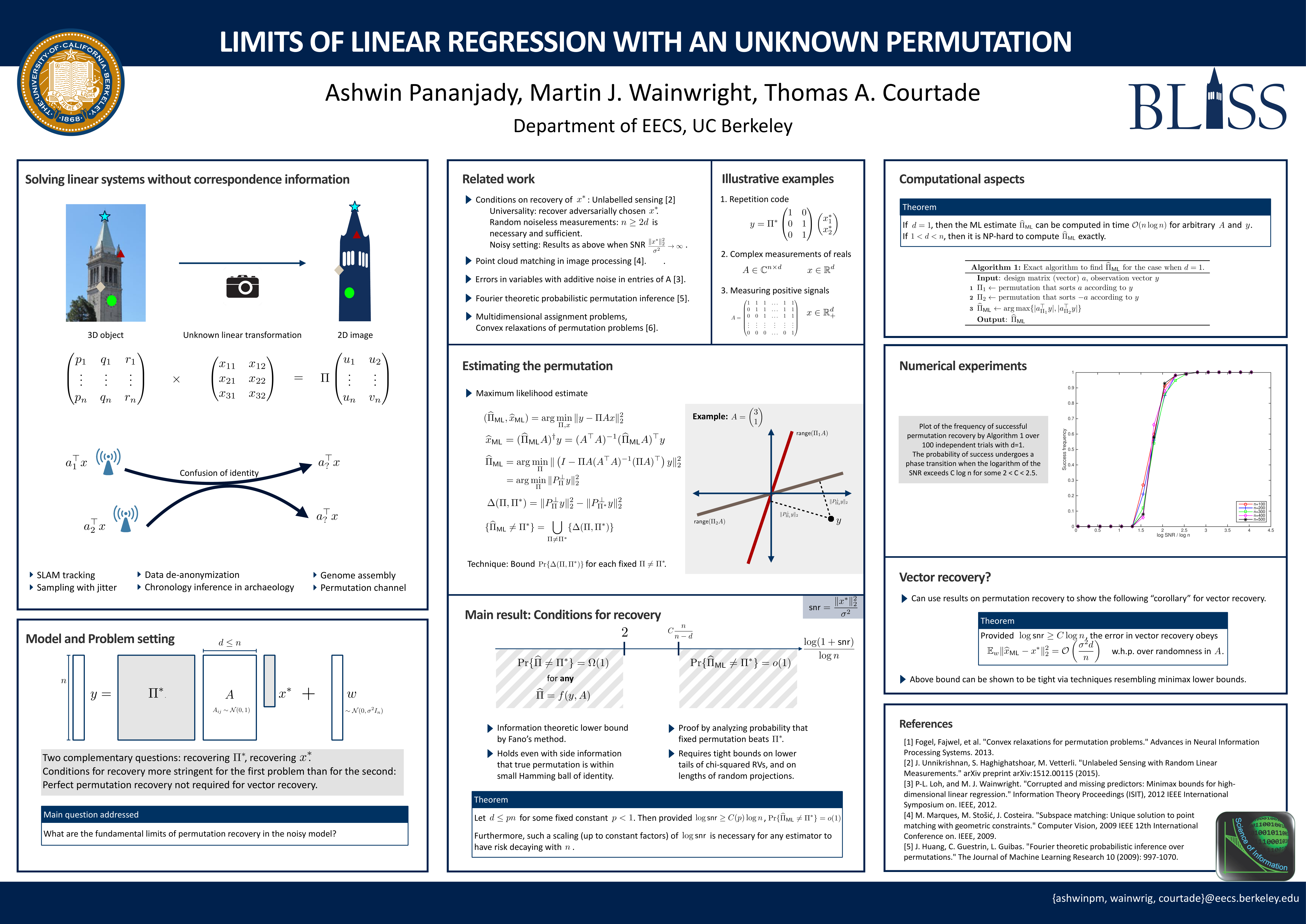}
    \caption{Example of pose and correspondence estimation. The camera introduces an unknown linear transformation corresponding to the pose. The unknown permututation represents the correspondence between points, which is shown in the picture via coloured shapes, and needs to be estimated. } \label{fig:pose}
    \end{center}
\end{figure}


The discrete analog of the model~\eqref{obsmodel} in which the vectors $x^*$ and $y$, and the matrix $A$ are all constrained to belong to some finite alphabet/field corresponds to the permutation channel studied by Schulman and Zuckerman \cite{schulman1999asymptotically}, with $A$ representing the (linear) encoding matrix. However, techniques for the discrete problem do not carry over to the continuous problem~\eqref{obsmodel}. 

Another line of work that is related in spirit to the observation model~\eqref{obsmodel} is the genome assembly problem from shotgun reads \cite{huang1999cap3}, in which an underlying vector $ x^* \in \{A,T,G, C\}^d$ must be assembled from an unknown permutation of its continuous sub-vector measurements, called ``reads". Two aspects, however, render it a particularization of our observation model, besides the obvious fact that $x^*$ in the genome assembly problem is constrained to a finite alphabet: (i)~in genome assembly, the matrix $A$ is fixed and consists of shifted identity matrices that select sub-vectors of $ x^*$, and (ii) the permutation matrix of genome assembly is in fact a block permutation matrix that permutes sub-vectors instead of coordinates as in equation~\eqref{obsmodel}. 

\subsection{Related work}
Previous work related to the observation model \eqref{obsmodel} can be broadly classified into two categories -- those that focus on $x^*$ recovery, and those focussed on recovering the underlying permutation. We discuss the most relevant results below.

\subsubsection{Latent vector estimation}
The observation model~\eqref{obsmodel} appears in the context of compressed sensing with an unknown sensor permutation \cite{emiya2014compressed}. The authors consider the matrix-based observation model $Y = \Pi^* AX^* + W$, where $X^*$ is a matrix whose columns are composed of multiple unknown vectors. Their contributions include a branch and bound algorithm to recover the underlying $X^*$, which they show to perform well empirically for small instances under the setting in which the entries of the matrix $A$ are drawn i.i.d. from a Gaussian distribution.

In the context of pose and correspondence estimation, the paper \cite{marques2009subspace} considers the noiseless observation model~\eqref{obsmodel}, and shows that if the permutation matrix maps a sufficiently large number of positions to themselves, then $x^*$ can be recovered reliably.

In the context of molecular channels, the model~\eqref{obsmodel} has been analyzed for the case when $x^*$ is some random vector, $A = I$, and $w$ represents non-negative noise that models delays introduced between emitter and receptor. Rose et al. \cite{rose1} provide lower bounds on the capacity of such channels. In particular, their results yield closed-form lower bounds for some special noise distributions, e.g., exponentially random noise.

A more recent paper \cite{unl} that is most closely related to our model considers the question of when the equation~\eqref{obsmodel} has a unique solution $x^*$, i.e., the identifiability of the noiseless model. The authors show that if the entries of $A$ are sampled i.i.d. from any continuous distribution with $n \geq 2d$, then equation~\eqref{obsmodel} has a unique solution $x^*$ with probability $1$. They also provide a converse showing that if $n < 2d$, any matrix $A$ whose entries are sampled i.i.d. from a continuous distribution does not (with probability $1$) have a unique solution $x^*$ to equation~\eqref{obsmodel}. While the paper shows uniqueness, the question of designing an efficient algorithm to recover a solution, unique or not, is left open. The paper also analyzes the stability of the noiseless solution, and establishes that $x^*$ can be recovered exactly when the SNR goes to infinity.

We also briefly compare the model~\eqref{obsmodel} with the problem of vector recovery in unions of subspaces, studied widely in the compressive sensing literature \cite{lu2008theory, blumensath2011sampling}. In the compressive sensing setup, the vector $x^*$ lies in the union of finitely many subspaces, and must be recovered from linear measurements with a random matrix, without a permutation. In our model, on the other hand, the vector $x^*$ is unrestricted, and the observation $y$ lies in the union of $n!$ subspaces -- one for each permutation. While the two models share a superficial connection, results do not carry over from one to the other in any obvious way. In fact, our model is fundamentally different from traditional compressive sensing, since the unknown permutation acts on the \emph{row space} of the design matrix $A$. In contrast, restricting $x^*$ to a union of subspaces (or more specifically, restricting its sparsity) influences the column space of $A$.
\subsubsection{Latent permutation estimation}
While our paper seems to be the first to consider permutation recovery in the linear regression model \eqref{obsmodel}, there are many related problems for which permutation recovery has been studied. We mention only those that are most closely related to our work.

The problem of feature matching in machine learning \cite{collier2016minimax} bears a superficial resemblance to our observation model. There, observations take the form $Y = X^* + W$ and $Y' = \Pi^* X^* + W'$, with all of $(X^*, Y, Y', W, W')$ representing matrices of appropriate dimensions, and the goal is to recover $\Pi^*$ from the tuple $(Y, Y')$. The paper \cite{collier2016minimax} establishes minimax rates on the separation between the rows of $X^*$ (as a function of problem parameters $n, d, \sigma$) that allow for exact permutation recovery.

The problem of statistical seriation \cite{rigollet} involves an observation model of the form $Y = \Pi^* X^* + W$, with the matrix $X^*$ obeying some shape constraint. In particular, if the columns of $X^*$ are unimodal (or, as a special case, monotone), then Flammarion et al. \cite{rigollet} establish minimax rates for the problem in the prediction error metric $\|\Pihat \widehat{X} - \Pi^* X^* \|_F^2$ by analyzing the least squares estimator. The seriation problem was also considered by Fogel et al. \cite{fogel2013convex} in the context of designing convex relaxations to permutation problems.

Permutation estimation has also been considered in other observation models involving matrices with structure, particularly in the context of ranking \cite{nihar, chatterjee}, or even more generally, in the context of \emph{identity management} \cite{huang2009fourier}. While we mention both of these problems because are related in spirit to permutation recovery, the problem setups do not bear too much resemblance to our linear model \eqref{obsmodel}.

Algorithmic approaches to solving for $\Pi^*$ in equation~\eqref{obsmodel} are related to the multi-dimensional assignment problem. In particular, while finding the correct permutation mapping between two vectors minimizing some loss function between them corresponds to the 1-dimensional assignment problem, here we are faced with an assignment problem between subspaces. While we do not elaborate on the vast literature that exists on solving variants on assignment problems, we note that broadly speaking, assignment problems in higher dimensions are much harder than the 1-D assignment problem. A survey on the quadratic assignment problem \cite{loiola2007survey} and references therein provide examples and methods that are currently used to solve these problems.


\subsection{Contributions} 
Our primary contribution addresses permutation recovery in the noisy version of observation model~\eqref{obsmodel}, with a random design matrix $A$. In particular, when the entries of $A$ are drawn i.i.d. from a standard Gaussian matrix, we show sharp conditions on the SNR under which exact permutation recovery is possible. We also derive necessary conditions for approximate permutation recovery to within a prescribed Hamming distortion.

We also briefly address the computational aspect of the permutation recovery problem. We show that the information theoretically optimal estimator we propose for exact permutation recovery is NP-hard to compute in the worst case. For the special case of $d= 1$, however, we show that it can be computed in polynomial time. Our results are corroborated by numerical simulations.

\subsection{Organization} 
The paper is organized as follows. In the next section, we set up notation and formally state the problem. In Section~\ref{sec:mainresult}, we state our main results and discuss some of their implications. We provide proofs of the main results in Section~\ref{sec:proofs}, deferring the more technical lemmas to the appendices. 

\section{Background and problem setting}
In this section, we set up notation, state the formal problem, and provide concrete examples of the noiseless version of our observation model by considering some fixed design matrices.

\subsection{Notation}

Since most of our analysis involves metrics involving permutations, we introduce all the relevant notation in this section. Permutations are denoted by $\pi$ and permutation matrices by $\Pi$. We use $\pi(i)$ to denote the image of an element $i$ under the permutation $\pi$. With a minor abuse of notation, we let $\mathcal{P}_n$ denote both the set of permutations on $n$ objects as well as the corresponding set of permutation matrices. We sometimes use the compact notation $y_\pi$ (or~$y_\Pi$) to denote the vector $y$ with entries permuted according to the permutation $\pi$ (or $\Pi$).

We let $\dH(\pi, \pi')$ denote the Hamming distance between two permutations. More formally, we have ${\dH(\pi, \pi') := \#\{ i \mid \pi(i) \neq \pi'(i)\}}$. For convenience, we let $\dH(\Pi, \Pi')$ denote the Hamming distance between two permutation matrices, which is to be interpreted as the Hamming distance between the corresponding permutations. 

The notation $v_i$ denotes the $i$th entry of a vector $v$. We denote the $i$th standard basis vector in $\mathbb{R}^d$ by $e_i$. We use the notation $a_i^\top$ to refer to the $i$th row of $A$. We also use the standard shorthand notation $[n]:=\{1, 2, \ldots, n \}$.

We also make use of standard asymptotic $\BigO$ notation. Specifically, for two real sequences $f_n$ and $g_n$, $f_n = \BigO(g_n)$ means that $f_n \leq C g_n$ for a universal constant $C>0$. Lastly, all logarithms denoted by $\log$ are to the base $e$, and we use $c_1, c_2$, etc. to denote absolute constants that are independent of other problem parameters.

\subsection{Formal problem setting and permutation recovery} \label{sec:setting}

As mentioned in the introduction, we focus exclusively on the noisy observation model in the random design setting. In other words, we obtain an $n$-vector of observations $y$ from the model~\eqref{obsmodel} with $n \geq d$ to ensure identifiability, and with the following assumptions: 
\paragraph*{Signal model} The vector $x^* \in \real^d$ is fixed, but unknown. We note that this is different from the \emph{adversarial} signal model of Unnikrishnan et al. \cite{unl}, and we provide clarifying examples in Section \ref{sec:examples}.

\paragraph*{Measurement matrix} The measurement matrix $A \in \real^{n \times d}$ is a random matrix of i.i.d. standard Gaussian variables chosen without knowledge of $x^*$. Our assumption on i.i.d. standard Gaussian designs easily extends to accommodate the more general case when rows of $A$ are drawn i.i.d. from the distribution $\NORMAL(0, \Sigma)$. In particular, writing $A = W \sqrt{\Sigma}$, where $W$ in an $n \times d$ standard Gaussian matrix and $\sqrt{\Sigma}$ denotes the symmetric square root of the (non-singular) covariance matrix $\Sigma$, our observation model takes the form
$$y = \Pi^* W \sqrt{\Sigma} x^* + w,$$
and the unknown vector is now $\sqrt{\Sigma}x^*$ in the model \eqref{obsmodel}.

\paragraph*{Noise variables} The vector $w \sim \NORMAL(0, \sigma^2 I_n)$ represents uncorrelated noise variables, each of (possibly unknown) variance $\sigma^2$. As will be made clear in the analysis, our assumption that the noise is Gaussian also readily extends to accommodate i.i.d. $\sigma$-sub-Gaussian noise.
Additionally, the permutation noise represented by the unknown permutation matrix $\Pi^*$ is arbitrary.


The main recovery criterion addressed in this paper is that of exact permutation recovery, which is formally described below. Following that, we also discuss two other relevant recovery criteria.

\paragraph*{Exact permutation recovery} The problem of exact permutation recovery is to recover $\Pi^*$, and the risk of an estimator is evaluated on the $0$-$1$. 
More formally, given an estimator of $\Pi^*$ denoted by $\Pihat : (y, A) \to \mathcal{P}_n$, we evaluate its risk by
\begin{align}
\Rpihat = \EE \left[ \mathbf{1} \{\Pihat \neq \Pi^*\} \right],
\end{align}
where the probability in the LHS is taken over the randomness in $y$ induced by both $A$ and $w$.

\paragraph*{Approximate permutation recovery} It is reasonable to think that recovering $\Pi^*$ up to some distortion is sufficient for many applications. Such a relaxation of exact permutation recovery allows the estimator to output a $\Pihat$ such that $\dH(\Pihat, \Pi^*) \leq D$, for some distortion $D$ to be specified. The risk of such an estimator is again evaluated on the $0$-$1$ loss of this error metric, given by
$\Pr\{\dH(\Pihat, \Pi^*) \geq D\}$, 
with the probability again taken over both $A$ and $w$.
While our results are derived mainly in the context of exact permutation recovery, they can be suitably modified to also yields results for approximate permutation recovery.

\paragraph*{Recovery with side information} In this variation, the unknown permutation matrix is not arbitrary, but known to be in some Hamming ball around the identity matrix. In other words, the estimator is provided with side information that $\dH(\Pi^*, I) \leq \bar{h}$, for some $\bar{h} < n$. In many applications, this may constitute a prior that leads us to believe that the permutation matrix is not arbitrary. In multi-target tracking, for example, we may be sure that at any given time, a certain number of measurements correspond to the true sensors that made them (that are close to the target, perhaps). Our results also address the exact permutation recovery problem with side information.
\newline

We now briefly give some examples in which the noiseless version is identifiable. 

\subsection{Illustrative examples of the noiseless model} \label{sec:examples}
In this section, we present two examples to illustrate the problem of permutation recovery and highlight the difference between our signal model and that of Unnikrishnan et al. \cite{unl}.

\paragraph*{Example 1} Consider the noiseless case of the observation model \eqref{obsmodel}. Let ${\nu_i, \nu'_i \;\; (i = 1, 2, \ldots, d)}$ represent i.i.d. continuous random variables, and form the design matrix $A$ by choosing
\begin{align}
a_{2i -1}^\top := \nu_i e_i^\top \text{ and } a_{2i}^\top = \nu'_i e_i^\top, \;\;\; i = 1, 2, \ldots, d. \nonumber
\end{align}
Note that $n = 2d$.
Now consider our fixed but unknown signal model for $x^*$. Since the permutation is arbitrary, our observations can be thought of as the unordered set ${\{\nu_i x^*_i, \nu'_i x^*_i \mid i \in [d] \}}$. With probability $1$, the ratios $r_i := \nu_i / \nu'_i$ are distinct for each $i$, and also that $\nu_i x^*_i \neq \nu_j x^*_j$ with probability $1$, by assumption of a fixed $x^*$. Therefore, there is a one to one correspondence between the ratios $r_i$ and $x^*_i$. All ratios are computable in time $\BigO(n^2)$, and $x^*$ can be exactly recovered. Using this information, we can also exactly recover $\Pi^*$. 

\paragraph*{Example 2} A particular case of this example was already observed by Unnikrishnan et al. \cite{unl}, but we include it to illustrate the difference between our signal model and the adversarial signal model.
Form the fixed design matrix $A$ by including $2^{i-1}$ copies of the vector $e_i$ among its rows. We therefore\footnote{Unnikrishnan et al. \cite{unl} proposed that $e_i$ be repeated $i$ times, but it is easy to see that this does not ensure recovery of an adversarially chosen $x^*$.} have $n = \sum_{i=1}^d 2^{i-1} = 2^d - 1$.

Our observations therefore consist of $2^{i-1}$ repetitions of $x^*_i$ for each $i \in [d]$. The value of $x^*_i$ can therefore be recovered by simply counting the number of times it is repeated, with our choice of the number of repetitions also accounting for cases when $x^*_i = x^*_j$ for some $i \neq j$. Notice that we can now recover \emph{any} vector $x^*$, even those chosen adversarially with knowledge of the $A$ matrix. Therefore, such a design matrix allows for an \emph{adversarial} signal model, in the flavor of compressive sensing \cite{candesCS}.

 Having provided examples of the noiseless observation model, we now return to the noisy setting of Section \ref{sec:setting}, and state our main results.

\section{Main results} \label{sec:mainresult}
In this section, we state our main theorems and discuss their consequences. Proofs of the theorems can be found in Section~\ref{sec:proofs}.

\subsection{Statistical limits of exact permutation recovery} \label{sec:statresult}
Our main theorems in this section provide necessary and sufficient conditions under which the probability of error in exactly recovering the true permutation goes to zero. 

In brief, provided that $d$ is sufficiently small, we establish a threshold phenomenon that characterizes how the signal-to-noise ratio $\snr: =\SNR$ must scale relative to $n$ in order to ensure identifiability.   More specifically, defining the ratio
\begin{align}
\Gamma\left( n, \snr \right) := \frac{\log \left(1 + \snr \right)}{\log n}, \nonumber
\end{align}
we show that the maximum likelihood estimator recovers the true permutation with high probability provided $\Gamma(n, \snr) \gg c$, where $c$ denotes an absolute constant.  Conversely,   if $\Gamma(n, \snr) \ll c$, then exact permutation recovery is impossible. 
For illustration, we have plotted the behaviour of the maximum likelihood estimator for the case when $d=1$ in Figure~\ref{sim}. Evidently, there is a sharp phase transition between error and exact recovery as the ratio $\Gamma(n, \snr)$ varies from 3 to 5. 

\begin{figure}
    \begin{center}
    \includegraphics[scale=.55]{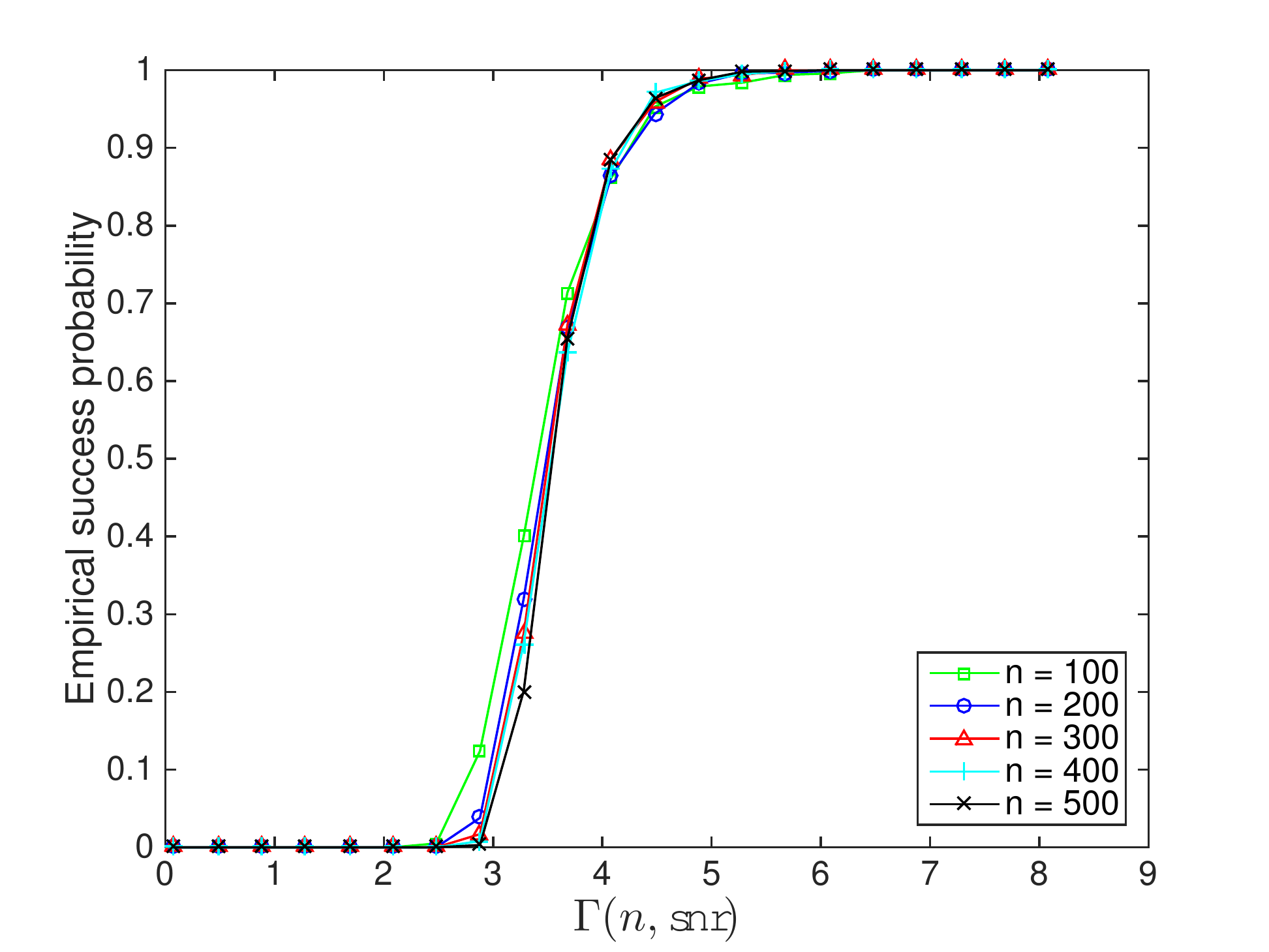}
    \caption{Empirical frequency of the event  $\{\Pihatml  = \Pi^*\}$ over $1000$ independent trials with ${d=1}$, plotted against $\Gamma\left( n, \snr \right)$ for different values of $n$. The probability of successful permutation recovery undergoes a phase transition as $\Gamma\left( n, \snr \right)$ varies from 3 to 5. This  is consistent with the prediction of Theorems~\ref{thm:upperbound} and~\ref{thm:lowerbound}.} \label{sim}
    \end{center}
\end{figure}

Let us now turn to more precise statements of our results.
We first define the maximum likelihood estimator (MLE) as
\begin{align}
(\Pihat_{{\sf ML}}, \xhat_{{\sf ML}}) = \arg\min_{\substack{\Pi \in \mathcal{P}_n \\ x \in \real^d}} \| y - \Pi A x\|^2_2.  \label{eq:ML}
\end{align}

The following theorem provides an upper bound on the probability of error of $\Pihatml$, with $(c_1, c_2)$ denoting absolute constants. 
\begin{theorem} \label{thm:upperbound}
For any $d<n$ and $\epsilon < \sqrt{n}$, if
\begin{align}
\log \left(\SNR\right) \geq \left( c_1 \frac{n}{n -d} + \epsilon\right) \log n, \label{exactcondn}
\end{align} 
then 
$\Rpihatml \leq c_2 n^{-2\epsilon}$. 
%
%
\end{theorem}

Theorem~\ref{thm:upperbound} provides conditions on the signal-to-noise ratio $\snr = \SNR$ that are sufficient for permutation recovery in the non-asymptotic, noisy regime. In contrast, the results of Unnikrishnan et al. \cite{unl} are stated  in the limit $\snr \to \infty$, without an explicit characterization of the scaling behavior.

We also note that Theorem~\ref{thm:upperbound} holds for all values of $d<n$, whereas the results of Unnikrishnan et al. \cite{unl} require $n \geq 2d$ for identifiability of $x^*$ in the noiseless case. Although the recovery of $\Pi^*$ and $x^*$ are not directly comparable, it is worth pointing out that the discrepancy also arises due to the difference between our fixed and unknown signal model, and the adversarial signal model assumed in the paper \cite{unl}.

We now turn to the following converse result, which complements Theorem \ref{thm:upperbound}.

\begin{theorem}\label{thm:lowerbound}
For any $\delta \in (0,2)$, if 
\begin{align}
2 + \log \left(1+ \SNR\right)\leq (2 - \delta) \log n, \label{lbcondn}
\end{align}
then 
$\Rpihat \geq 1 - c_3 e^{-c_4 n\delta}$ for any estimator $\Pihat$. 
\end{theorem}
Theorem~\ref{thm:lowerbound} serves as a ``strong converse" for our problem, since it guarantees that if condition~\eqref{lbcondn} is satisfied, then the probability of error of any estimator goes to $1$ as $n$ goes to infinity. Indeed, it is proved using the strong converse argument for the Gaussian channel~\cite{shannon1959probability}.
In fact, we are also able to show the following ``weak converse" in the presence of side information.
\begin{proposition} \label{rem:lowerbound}
If $n \geq 9$ and 
\begin{align}
\log \left(1 + \SNR\right)\leq \frac{8}{9} \log \left( \frac{n}{8} \right), \nonumber
\end{align} then $\Rpihat \geq 1/2$ for any estimator $\Pihat$,  even if it is known a-priori that $\dH(\Pi^*, I) \leq 2$.
\end{proposition}

As mentioned earlier,  restriction of $\Pi^*$ constitutes some application-dependent prior; the strongest such prior restricts it to a Hamming ball of radius $2$ around the identity.  Proposition~\ref{rem:lowerbound}  asserts that even this side information {does not} substantially change the statistical limits of permutation recovery. 





It is also worth noting that the converse results hold uniformly over $d$. In particular, if $d \leq pn$ for some fixed $p < 1$,  Theorems~\ref{thm:upperbound} and~\ref{thm:lowerbound} together yield the threshold behavior of identifiability in terms of  $\Gamma(n, \snr)$ that was discussed above.
In the next section, we find  that a similar phenomenon occurs even with approximate permutation recovery.

\subsection{Limits of approximate permutation recovery}
The techniques we used to prove results for exact permutation recovery can be suitably modified to obtain results for approximate permutation recovery to within a Hamming distortion~$D$. 
In particular, we show the following converse result for approximate recovery.
\begin{theorem} \label{thm:approx}
For any $2 <  D \leq n-1$, if
\begin{align}
\log \left( 1 + \SNR \right) \leq \frac{n - D+1}{n} \log \left( \frac{n- D+1}{2e} \right),
\end{align}
then $\Pr\{\dH(\Pihat, \Pi^*) \geq D\} \geq 1/2$ for any estimator $\Pihat$.
\end{theorem}
Note that for any $D \leq pn$ with $p \in (0,1)$, Theorems \ref{thm:upperbound} and \ref{thm:approx} provide a set of sufficient and necessary conditions for approximate permutation recovery that match up to constant factors. In particular, the necessary condition resembles that for exact permutation recovery, and the same SNR threshold behaviour is seen even here.
 We remark that a corresponding converse with side information can also be proved for approximate permutation recovery using techniques similar to the proof of Proposition \ref{rem:lowerbound}. It is also worth mentioning the following:

\begin{remark} \label{rem:sideinfo}
The converse results given by Theorem \ref{thm:lowerbound}, Proposition \ref{rem:lowerbound}, and Theorem \ref{thm:approx} hold even when the estimator has exact knowledge of $x^*$.
\end{remark}






\subsection{Computational aspects} \label{sec:compresult}

In the previous sections, we considered the MLE given by equation~\eqref{eq:ML} and analyzed its statistical properties. However, since equation \eqref{eq:ML} involves a combinatorial minimization over $n!$ permutations, it is unclear if $\Pihatml$ can be computed efficiently. The following theorem addresses this question.

\begin{theorem} \label{thm:1D}
For $d=1$, the MLE $\;\Pihatml$ can be computed in time $\BigO(n\log n)$ for any choice of the measurement matrix $A$.
In contrast, if $d >1$, then $\Pihatml$ is NP-hard to compute.
\end{theorem}
The algorithm used to prove the first part of the theorem involves a simple sorting operation, which introduces the $\BigO(n\log n)$ complexity. We emphasize that the algorithm assumes no prior knowledge about the distribution of the data; for every given $A$ and $y$, it returns the optimal solution to problem~\eqref{eq:ML}.

The second part of the theorem  asserts that the algorithmic simplicity enjoyed by the $d = 1$ case does not extend to general $d$. The proof proceeds by a reduction from the NP-complete partition problem.
%
We stress here that the NP-hardness claim holds over worst case input instances. In particular, it does not preclude the possibility that there exists a polynomial time algorithm that solves problem~\eqref{eq:ML} with high probability when $A$ is chosen randomly as in our original setting. However, we conjecture that solving problem~\eqref{eq:ML} over random $A$ is also a computationally hard problem, conditioned on an average-case hardness assumption.

\section{Proofs of Main Results} \label{sec:proofs}

In this section, we prove our main results.  Technical details are deferred to the appendices.
Throughout the proofs, we assume that $n$ is larger than some universal constant. The case where $n$ is smaller can be handled by changing the constants in our proofs appropriately. We also use the notation $c, c'$ to denote absolute constants that can change from line to line.

We begin with the proof of Theorem~\ref{thm:upperbound}. At a high level, it involves bounding the probability that any fixed permutation is preferred to $\Pi^*$ by the estimator. The analysis requires precise control on the lower tails of \mbox{$\chi^2$-random} variables, and tight bounds on the norms of random projections, for which we use results derived in the context of dimensionality reduction by Dasgupta and Gupta \cite{dasguptajl}.

In order to simplify the exposition, we first consider the case when $d = 1$ in Section~\ref{sec:proofub1}, and later make the necessary modifications for the general case in Section~\ref{sec:proofub2}. 

\subsection{Proof of Theorem~\ref{thm:upperbound}: $d=1$ case} \label{sec:proofub1}
Recall the definition of the maximum likelihood estimator 
\begin{align}
(\Pihat_{{\sf ML}}, \xhat_{{\sf ML}}) = \arg\min_{\Pi \in \mathcal{P}_n}\min_{x \in \real^d} \| y - \Pi A x\|^2_2. \nonumber
\end{align}
For a fixed permutation matrix $\Pi$, assuming that $A$ has full column rank\footnote{An $n \times d$ i.i.d. Gaussian random matrix has full column rank with probability $1$ as long as $d \leq n$}, the minimizing argument $x$ is simply $(\Pi A)^\dagger y$, where $X^\dagger = (X^\top X)^{-1} X^\top$ represents the pseudoinverse of a matrix $X$.
By computing the minimum over $x \in \real^d$ in the above equation, we find that the maximum likelihood estimate of the permutation is given by
\begin{align}
\Pihatml = \arg\min_{\Pi \in \mathcal{P}_n} \| \Projorthpi y \|^2_2, \label{eq:MLperm}
\end{align}
where $\Projorthpi = I - \Pi A (A^\top A)^{-1} (\Pi A)^\top$ denotes the projection onto the orthogonal complement of the column space of $\Pi A$.

For a fixed $\Pi \in \mathcal{P}_n$, define the random variable
\begin{align}
\Delta(\Pi, \Pi^*) := \|\Projorthpi y\|_2^2 - \|\Projorthpistar y\|^2_2. \label{zpi}
\end{align}
For any permutation $\Pi$, the estimator~\eqref{eq:MLperm} prefers the permutation $\Pi$ to $\Pi^*$ if $\Delta(\Pi, \Pi^*) \leq 0$. The overall error event occurs when $\Delta(\Pi, \Pi^*) \leq 0$ for some $\Pi$, meaning that
\begin{align}
\{ \Pihatml \neq \Pi^*\} =  \bigcup_{\Pi \in \mathcal{P}_n \setminus {\Pi^*}} \{\Delta(\Pi, \Pi^*) \leq 0\}. \label{individual}
\end{align}

Equation~\eqref{individual} holds for any value of $d$. We shortly specialize to the $d=1$ case. Our strategy for proving Theorem~\ref{thm:upperbound} boils down to  bounding the probability of each error event in the RHS of equation~\eqref{individual} using the following key lemma, proved in Section~\ref{sec:keylemma1}. Recall the definition of $\dH(\Pi, \Pi')$, the Hamming distance between two permutation matrices.



\begin{lemma} \label{lem:keylemma1}
For $d = 1$ and any two permutation matrices $\Pi$ and $\Pi^*$, and provided $\SNR > 1$, we have
$$\Pr \{\Delta(\Pi, \Pi^*) \leq 0 \} \leq c' \exp \left(-c\, \dH(\Pi, \Pi^*) \log \left(\SNR \right)\right).$$
\end{lemma}


We are now ready to prove Theorem \ref{thm:upperbound}.


\begin{proof}[Proof of Theorem \ref{thm:upperbound} for $d=1$] Fix $\epsilon>0$ and assume that the following consequence of condition \eqref{exactcondn} holds:
\begin{align}
c \log \left(\SNR \right) \geq (1 + \epsilon) \log n, \label{condnsnr}
\end{align}
where $c$ is the same as in Lemma \ref{lem:keylemma1}.  Now, observe that
\begin{align}
\Rpihatml &\leq \sum_{\Pi \in \mathcal{P}_n \setminus \Pi^*} \Pr \{\Delta(\Pi, \Pi^*) \leq 0 \} \nonumber \\
&\stackrel{\1}{\leq} \sum_{\Pi \in \mathcal{P}_n \setminus \Pi^*} c' \exp \left(-c\, \dH(\Pi, \Pi^*) \log \left(\SNR \right) \right) \nonumber \\
&\leq c' \sum_{2 \leq k \leq n} n^k \exp \left(-c\, k \log \left(\SNR \right) \right) \nonumber\\
&\stackrel{\2}{\leq} c' \sum_{2 \leq k \leq n} n^{-\epsilon k} \nonumber \\
&\leq c' \frac{1}{n^\epsilon (n^\epsilon - 1)}. \nonumber
\end{align}
where step $\1$ follows since $\#\{ \Pi : \dH(\Pi, \Pi^*)=k\}\leq n^k$, and step $\2$ follows from condition~\eqref{condnsnr}. 
Relabelling the constants in condition~\eqref{condnsnr} proves the theorem.
\end{proof}

\subsubsection{Proof of Lemma~\ref{lem:keylemma1}} \label{sec:keylemma1}
Before the proof, we establish notation. 
For each $\delta>0$, define the events
\begin{subequations}
\begin{align}
\Fspace_1 (\delta) &= \left \{ |\|\Projorthpistar y\|^2_2  - \|\Projorthpi w\|^2_2| \geq \delta \right\}, \text{ and} \\
\Fspace_2 (\delta) &= \left \{ \|\Projorthpi y\|^2_2  - \|\Projorthpi w\|^2_2 \leq 2\delta \right\}.\end{align}
\end{subequations}
Evidently, 
\begin{align}
\{\Delta(\Pi, \Pi^*) \leq 0 \} \subseteq \Fspace_1 (\delta) \cup \Fspace_2(\delta). \label{zpclaim}
\end{align}
Indeed, if neither  $\Fspace_1(\delta)$ nor $\Fspace_2(\delta)$ occurs 
\begin{align}
\Delta(\Pi, \Pi^*) = \left(\|\Projorthpi y\|^2_2  - \|\Projorthpi w\|^2_2\right) - \left(\|\Projorthpistar y\|^2_2  - \|\Projorthpi w\|^2_2 \right)
> 2\delta - \delta  = \delta. \nonumber
\end{align} 

Thus, to prove Lemma~\ref{lem:keylemma1}, we shall bound the probability of the two events $\Fspace_1(\delta)$ and $\Fspace_2(\delta)$ individually, and then invoke the union bound. Note that inequality~\eqref{zpclaim} holds for all values of $\delta >0$; it is convenient to choose $\delta^* := \frac{1}{3}\|\Projorthpi \Pi^* A x^* \|_2^2$. With this choice, the following lemma, proved in Appendix~\ref{app:noiselemma},  bounds the probabilities of the individual events over randomness in $w$ conditioned on a given $A$.

\begin{lemma} \label{lem:noise}
For any $\delta>0$ and with $\delta^* = \frac{1}{3}\|\Projorthpi \Pi^* A x^* \|_2^2$, we have
\begin{subequations}
\begin{align}
\Pr\nolimits_w\{ \Fspace_1(\delta) \} &\leq c' \exp\left(-c \frac{\delta}{\sigma^2} \right) \text{, and} \label{noiseterm} \\
\Pr\nolimits_w\{ \Fspace_2(\delta^*) \} &\leq c' \exp \left(-c \frac{\delta^*}{\sigma^2} \right)\label{signalterm}.
\end{align}
\end{subequations}
\end{lemma}


The next lemma, proved in Section~\ref{sec:tpi1}, is   needed in order to incorporate the randomness in $A$ into the required tail bound. It is convenient to introduce the shorthand $T_\Pi := \|\Projorthpi \Pi^* A x^* \|_2^2$. 

\begin{lemma}\label{lem:tpi1}
For $d= 1$ and any two permutation matrices $\Pi$ and $\Pi^*$ at Hamming distance $h$, we have
\begin{align}
\Pr\nolimits_A \{ T_\Pi \leq t \|x^*\|_2^2 \} \leq 6  \exp\left( - \frac{h}{10} \left[\log \frac{h}{t} + \frac{t}{h} -1 \right]\right) \label{eq:maxterm}
\end{align}
for all $t \in [0, h]$.
\end{lemma}

\noindent  We  now have all the ingredients  to prove Lemma \ref{lem:keylemma1}.

\begin{proof}[Proof of Lemma \ref{lem:keylemma1}]
Applying Lemma~\ref{lem:noise} and using the union bound then yields
\begin{align}
\Pr\nolimits_w \{\Delta(\Pi, \Pi^*) \leq 0\} &\leq \Pr\nolimits_w\{ \Fspace_1(\delta^*) \} + \Pr\nolimits_w\{ \Fspace_2(\delta^*) \} \nonumber\\
&\leq c' \exp \left(-c \frac{T_\Pi}{\sigma^2} \right). \label{eq:bound3}
\end{align}

Combining bound \eqref{eq:bound3} with Lemma \ref{lem:tpi1} yields
\begin{align}
\Pr \{\Delta(\Pi, \Pi^*) \leq 0\} &\leq c' \exp\left(-c \frac{t \|x^*\|_2^2}{\sigma^2}\right) \Pr\nolimits_A \{T_\Pi \geq t \|x^*\|_2^2\} + \Pr\nolimits_A \{T_\Pi \leq t \|x^*\|_2^2 \} \nonumber\\
&\leq c' \exp\left(-c \frac{t \|x^*\|_2^2}{\sigma^2}\right) + 6  \exp\left( - \frac{h}{10} \left[\log \frac{h}{t} + \frac{t}{h} -1 \right]\right), \label{balance}
\end{align}
where the last inequality holds provided that $t \in [0, h]$, and the probability in the LHS is now taken over randomness in both $w$ \emph{and} $A$.

Using the shorthand $\snr:= \SNR$, setting $t = h \frac{\log \snr}{\snr}$, and noting that $t \in [0, h]$ since $\snr > 1$, we have
\begin{align}
\Pr \{\Delta(\Pi, \Pi^*) \leq 0\} &\leq c' \exp\left(-c h \log \snr \right) + 6  \exp\left( - \frac{h}{10} \left[\log \left(\frac{\snr}{\log \snr}\right) + \frac{\log \snr}{\snr} - 1 \right]\right).
\end{align}
It is easily verified that for all $\snr >1$, we have
\begin{align}
\log \left(\frac{\snr}{\log \snr}\right) + \frac{\log \snr}{\snr} - 1 > \frac{\log \snr}{4}. \label{easy}
\end{align} 
Hence, after substituting for $\snr$, we have
\begin{align}
\Pr \{\Delta(\Pi, \Pi^*) \leq 0\} &\leq c' \exp\left(-c h \log \left( \SNR \right)\right).
\end{align}
\end{proof}

\subsubsection{Proof of Lemma~\ref{lem:tpi1}} \label{sec:tpi1}
In the case $d=1$, the matrix $A$ is composed of a single vector $a \in \real^n$. Recalling the random variable $T_\Pi = \|\Projorthpi \Pi^* A x^*\|_2^2$, we have
\begin{align}
T_\Pi &= (x^*)^2 \left( \|a\|_2^2 - \frac{1}{\|a \|_2^2} \langle a_\Pi, a \rangle^2 \right) \nonumber\\
&\stackrel{\1}{\geq} (x^*)^2 \left(\|a\|_2^2 - |\langle a, a_\Pi \rangle | \right) \nonumber\\
&= \frac{(x^*)^2}{2} \min \left(\| a - a_\Pi \|_2^2, \| a + a_\Pi \|_2^2 \right), \nonumber
\end{align}
where step $\1$ follows from the Cauchy Schwarz inequality. Applying the union bound then yields
\begin{align}
\Pr\{T_\Pi \leq t (x^*)^2 \} \leq \Pr\{\| a - a_\Pi \|_2^2 \leq 2t\} +  \Pr\{\| a + a_\Pi \|_2^2 \leq 2t\}. \nonumber
\end{align}

Let $Z_\ell$ and $\tilde{Z}_\ell$ denote (not necessarily independent) $\chi^2$ random variables with $\ell$ degrees of freedom. We split the analysis into two cases.

\paragraph{Case $h \geq 3$:} Lemma~\ref{lem:perm} from Appendix~\ref{sec:perm} guarantees that
\begin{subequations}
\begin{align}
\frac{\| a - a_\Pi \|_2^2}{2} &\stackrel{d}{=} Z_{h_1} + Z_{h_2} + Z_{h_3}, \text{ and} \\
\frac{\| a + a_\Pi \|_2^2}{2} &\stackrel{d}{=} \widetilde{Z}_{h_1} + \widetilde{Z}_{h_2} + \widetilde{Z}_{h_3} + \widetilde{Z}_{n-h},
\end{align}
\end{subequations}
where $\stackrel{d}{=}$ denotes equality in distribution and $h_1, h_2, h_3 \geq \frac{h}{5}$ with $h_1 + h_2 + h_3 = h$. An application of the union bound then yields
\begin{align}
\Pr\{\|a - a_\Pi\|_2^2 \leq 2t \} \leq \sum_{i=1}^3 \Pr \left\{Z_{h_i} \leq t \frac{h_i}{h} \right\}. \nonumber
\end{align}
Similarly, provided that $h \geq 3$, we have
\begin{align}
\Pr\{\|a + a_\Pi\|_2^2 \leq 2t \} &\leq \Pr\{\widetilde{Z}_{h_1} + \widetilde{Z}_{h_2} + \widetilde{Z}_{h_3} + \widetilde{Z}_{n-h} \leq t\} \nonumber\\
&\stackrel{\2}{\leq} \Pr\{\widetilde{Z}_{h_1} + \widetilde{Z}_{h_2} + \widetilde{Z}_{h_3} \leq t\} \nonumber\\
&\stackrel{\3}{\leq} \sum_{i=1}^3 \Pr \left\{\widetilde{Z}_{h_i} \leq t\frac{h_i}{h} \right\}, \nonumber
\end{align}
where inequality $\2$ follows from the non-negativity of $Z_{n - h}$, and the monotonicity of the CDF; and inequality $\3$ from the union bound.
Finally, bounds on the lower tails of $\chi^2$ random variables (see Lemma~\ref{lem:tailbound} in Appendix~\ref{sec:tailbound}) yield
\begin{align}
\Pr\left\{Z_{h_i} \leq t\frac{h_i}{h}\right\} = \Pr\left\{\widetilde{Z}_{h_i} \leq t\frac{h_i}{h}\right\} &\stackrel{\4}{\leq} \left( \frac{t}{h} \exp\left(1 - \frac{t}{h} \right)\right)^{h_i/2} \nonumber \\
&\stackrel{\5}{\leq} \left( \frac{t}{h} \exp\left(1 - \frac{t}{h} \right)\right)^{h/10}. \nonumber
\end{align}
Here, inequality $\4$ is valid provided $\frac{th_i}{h} \leq h_i$, or equivalently, if $t \leq h$, whereas inequality~$\5$ follows since $h_i \geq h/5$ and the function $x e^{1-x} \in [0,1]$ for all $x \in [0,1]$. Combining the pieces proves Lemma~\ref{lem:tpi1} for $h \geq 3$.

\paragraph{Case $h =2$:} In this case, we have
\begin{align}
\frac{\| a - a_\Pi \|_2^2}{2} \stackrel{d}{=} 2 Z_1, \;\; \text{ and }\;\; \frac{\| a + a_\Pi \|_2^2}{2} \stackrel{d}{=} 2\widetilde{Z}_1 +  \widetilde{Z}_{n-2}. \nonumber
\end{align}
Proceeding as before by applying the union bound and Lemma~\ref{lem:tailbound}, we have that for $t \leq 2$, the random variable $T_\Pi$ obeys the tail bound
\begin{align}
\Pr\{T_\Pi \leq t (x^*)^2 \} \leq 2 \left( \frac{t}{2} \exp\left(1 - \frac{t}{2} \right)\right)^{1/2} \leq 6 \left( \frac{t}{h} \exp\left(1 - \frac{t}{h} \right)\right)^{h/10}, \text{ for }h = 2. \nonumber
\end{align} \qed

In the next section, we prove Theorem~\ref{thm:upperbound} for the general case.

\subsection{Proof of Theorem~\ref{thm:upperbound}: Case $d \in \{2, 3, \ldots, n-1\}$} \label{sec:proofub2}
In order to be consistent, we follow the same proof structure as for the $d=1$ case. Recall the definition of $\Delta(\Pi, \Pi^*)$ from equation~\eqref{zpi}. We begin with an equivalent of the key lemma to bound the probability of the event $\{\Delta(\Pi, \Pi^*) \leq 0\}$. 

\begin{lemma} \label{lem:keylemma}
For any $1< d < n$, any two permutation matrices $\Pi$ and $\Pi^*$ at Hamming distance $h$, and provided $\left(\SNR \right) n^{-\frac{2n}{n-d}} > \frac{5}{4}$, we have
\begin{align}
\Pr \{\Delta(\Pi, \Pi^*) \leq 0\} &\leq c' \max \Bigg[ \exp \left(-n \log \frac{n}{2} \right), 
 \exp \left( c h \left( \log \left(\SNR \right) - \frac{2n}{n-d} \log n \right) \right) \Bigg].
\end{align}
\end{lemma}

We prove Lemma~\ref{lem:keylemma} in Section~\ref{pf:keylemma}. Taking it as given, we are ready to prove Theorem \ref{thm:upperbound} for the general case.

\begin{proof}[Proof of Theorem \ref{thm:upperbound}, general case]
As before, we use the union bound to prove the theorem. We begin by fixing some $\epsilon \in (0, \sqrt{n})$ and assuming that the following consequence of condition~\eqref{exactcondn} holds:
\begin{align}
c \log \left(\SNR \right) \geq \left(1 + \epsilon + c\frac{2n}{n-d}\right) \log n. \label{condn2snr}
\end{align}


Now define $b(k) := \sum_{\Pi : \dH(\Pi, \Pi^*) = k} \Pr \{\Delta(\Pi, \Pi^*) \leq 0\}$. Applying Lemma~\ref{lem:keylemma} then yields
\begin{align}
b(k) &\leq \frac{n!}{(n-k)!} c' \max \Bigg\{ \exp \left(-n \log \frac{n}{2} \right),  \exp \left( - c k \left(\log \left(\SNR \right) - \frac{2n}{n-d} \log n \right) \right) \Bigg\}. \label{max1}
\end{align}
We upper bound $b(k)$ by splitting the analysis into two cases.

\paragraph{Case 1:} If the first term attains the maximum in the RHS of inequality~\eqref{max1}, then for all $2 \leq k \leq n$, we have 
\begin{align}
b(k) 
&\leq c'n! \exp( -n \log n + n \log 2) \nonumber\\
&\stackrel{\1}{\leq} c'e\sqrt{n} \exp(-n \log n + n\log 2 + - n + n \log n) \nonumber\\ 
&\stackrel{\2}{\leq} \frac{c'}{n^{2\epsilon + 1}}, \nonumber
\end{align}
where inequality~$\1$ follows from the well-known upper bound $n! \leq e \sqrt{n}\left(\frac{n}{e}\right)^n$, and inequality~$\2$ holds since $\epsilon \in (0, \sqrt{n})$.

\paragraph{Case 2:} Alternatively, if the maximum is attained by the second term in the RHS of inequality~\eqref{max1}, then we have 
\begin{align}
b(k) &\leq n^k c' \exp \left( - c k \left( \log \left(\SNR \right) - \frac{2n}{n-d} \log n \right) \right) \nonumber\\
&\stackrel{\3}{\leq} c' n^{-\epsilon k}, \nonumber
\end{align}
where step $\3$ follows from condition \eqref{condn2snr}.

Combining the two cases, we have
\begin{align}
b(k) \leq \max\{c' n^{-\epsilon h}, cn^{-2\epsilon - 1} \} \leq \left(c' n^{-\epsilon h} + cn^{-2\epsilon - 1}\right). \nonumber
\end{align}

The last step is to use the union bound to obtain 
\begin{align}
\Rpihatml &\leq \sum_{2 \leq k \leq n} b(k) \label{eq:sumbk}\\ 
& \leq \sum_{2 \leq k \leq n} \left(c' n^{-\epsilon h} + cn^{-2\epsilon - 1}\right) \nonumber\\
&\stackrel{\4}{\leq} c n^{-2\epsilon}, \nonumber
\end{align}
where step $\4$ follows by a calculation similar to the one carried out for the $d=1$ case. Relabelling the constants in condition~\eqref{condn2snr} completes the proof.
\end{proof}


\subsubsection{Proof of Lemma~\ref{lem:keylemma}} \label{pf:keylemma}
The first part of the proof is exactly the same as that of Lemma~\ref{lem:keylemma1}. In particular, Lemma~\ref{lem:noise} applies without modification to yield a bound identical to the inequality~\eqref{eq:bound3}, given by
\begin{align}
{\Pr\nolimits_w} \{\Delta(\Pi, \Pi^*) \leq 0\} \leq c' \exp \left(-c \frac{T_\Pi}{\sigma^2} \right), \label{eq:bound3'}
\end{align}
where $T_\Pi = \|\Projorthpi \Pi^* A x^*\|_2^2$, as before. 

The major difference from the $d =1$ case is in the random variable $T_\Pi$. Accordingly, we state the following parallel lemma to Lemma~\ref{lem:tpi1}.

\begin{lemma}\label{lem:tpi}
For $1 < d < n$, any two permutation matrices $\Pi$ and $\Pi^*$ at Hamming distance $h$, and $t \leq hn^{-\frac{2n}{n -d}}$, we have
\begin{align}
\Pr\nolimits_A \{ T_\Pi \leq t \|x^*\|_2^2 \} \leq 2 \max \left\{ \exp \left(-n \log \frac{n}{2}\right),  6 \exp \left( -\frac{h}{10} \left[ \log \left(\frac{h}{tn^{\frac{2n}{n-d}}}\right) + \frac{tn^{\frac{2n}{n-d}}}{h} - 1 \right] \right)\right\}. \label{eq:maxterm2}
\end{align}
\end{lemma}
The proof of Lemma~\ref{lem:tpi} appears in Section~\ref{pf:tpi}. We are now ready to prove Lemma \ref{lem:keylemma}.

\begin{proof}[Proof of Lemma \ref{lem:keylemma}]
We prove Lemma~\ref{lem:keylemma} from Lemma~\ref{lem:tpi} and equation~\eqref{eq:bound3'} by an argument similar to the one before.
In particular, in a similar vein to the steps leading up to equation~\eqref{balance}, we have
\begin{align}
\Pr \{\Delta(\Pi, \Pi^*) \leq 0\} &\leq c' \exp\left(-c \frac{t \|x^*\|_2^2}{\sigma^2}\right) + \Pr\nolimits_A \{ T_\Pi \leq t \|x^*\|_2^2 \}. \label{balanced}
\end{align}
We now use the shorthand $\snr: = \SNR$ and
let $t^* = h\frac{\log \left(\snr \cdot n^{-\frac{2n}{n-d}}\right)}{\snr}$. Noting that $\snr \cdot n^{-\frac{2n}{n-d}} > 5/4$ yields $t^* \leq h n^{-\frac{2n}{n-d}}$, we set $t = t^*$ in inequality~\eqref{balanced} to obtain
\begin{align}
\Pr \{\Delta(\Pi, \Pi^*) \leq 0\} &\leq c' \exp\left(-c h \log sn^{-\frac{2n}{n-d}}\right) + \Pr\nolimits_A \{ T_\Pi \leq t^* \|x^*\|_2^2 \}. \label{balance2}
\end{align}


Since $\Pr_A \{ T_\Pi \leq t^* \|x^*\|_2^2 \}$ can be bounded by a maximum of two terms~\eqref{eq:maxterm2}, we now split the analysis into two cases depending on which term attains the maximum.

\paragraph{Case 1:} First, suppose that the second term attains the maximum in inequality~\eqref{eq:maxterm2}, i.e., $\Pr_A \{ T_\Pi \leq t^* \|x^*\|_2^2 \} \leq 12 \exp \left( -\frac{h}{10} \left[ \log \left(\frac{h}{t^*n^{\frac{2n}{n-d}}}\right) + \frac{t^*n^{\frac{2n}{n-d}}}{h} - 1 \right] \right)$. Substituting for $t^*$, we have

\begin{align}
\Pr\nolimits_A \{ T_\Pi \leq t^* \|x^*\|_2^2 \} \leq 12 \exp \left( -\frac{h}{10} \left[ \log \left( \frac{\snr \cdot n^{-\frac{2n}{n-d}}}{\log \left( \snr \cdot n^{-\frac{2n}{n-d}} \right)} \right) + \frac{\log \left( \snr \cdot n^{-\frac{2n}{n-d}} \right)}{\snr \cdot n^{-\frac{2n}{n-d}}} - 1 \right] \right).
\end{align}

We have $\snr \cdot n^{-\frac{2n}{n-d}} > \frac{5}{4}$, a condition which leads to the following pair of easily verifiable inequalities:
\begin{subequations}
\begin{align}
\log \left( \frac{\snr \cdot n^{-\frac{2n}{n-d}}}{\log \left( \snr \cdot n^{-\frac{2n}{n-d}} \right)} \right) + \frac{\log \left( \snr \cdot n^{-\frac{2n}{n-d}} \right)}{\snr \cdot n^{-\frac{2n}{n-d}}} - 1  &\geq \frac{\log \snr \cdot n^{-\frac{2n}{n-d}}}{4}, \text{ and} \label{easy3}\\
\log \left( \frac{\snr \cdot n^{-\frac{2n}{n-d}}}{\log \left( \snr \cdot n^{-\frac{2n}{n-d}} \right)} \right) + \frac{\log \left( \snr \cdot n^{-\frac{2n}{n-d}} \right)}{\snr \cdot n^{-\frac{2n}{n-d}}} - 1 &\leq 5 \log \left( \snr \cdot n^{-\frac{2n}{n-d}} \right). \label{easy2}
\end{align}
\end{subequations}

Using inequality \eqref{easy3}, we have
\begin{align}
\Pr\nolimits_A \{ T_\Pi \leq t^* \|x^*\|_2^2 \} \leq 12 \exp \left( -c h \log \left( \snr \cdot n^{-\frac{2n}{n-d}} \right)\right). \label{case1}
\end{align}
Now using inequalities~\eqref{case1} and~\eqref{balance2} together yields
\begin{align}
\Pr \{ \Delta(\Pi, \Pi^*) \leq 0 \} \leq c' \exp \left( -c h \log \left( \snr \cdot n^{-\frac{2n}{n-d}} \right)\right). \label{z1}
\end{align}

\noindent It remains to handle the second case.

\paragraph{Case 2:} Suppose now that $\Pr_A \{ T_\Pi \leq t^* \|x^*\|_2^2 \} \leq 2\exp \left(-n \log \frac{n}{2}\right)$, i.e., that the first term in RHS of inequality~\eqref{eq:maxterm2} attains the maximum when $t = t^*$. In this case, we have
\begin{align}
\exp \left(-n \log \frac{n}{2}\right) &\geq 6  \exp \left( -\frac{h}{10} \left[ \log \left(\frac{h}{t^*n^{\frac{2n}{n-d}}}\right) + \frac{t^*n^{\frac{2n}{n-d}}}{h} - 1 \right] \right) \nonumber \\
&\stackrel{\1}{\geq} c' \exp\left(-c h \log \left( \snr \cdot n^{-\frac{2n}{n-d}}\right) \right), \nonumber
\end{align}
where step $\1$ follows from the right inequality~\eqref{easy2}. Now substituting into inequality~\eqref{balance2}, we have
\begin{align}
\Pr \{\Delta(\Pi, \Pi^*) \leq 0\} &\leq c' \exp\left(-c h \log \left( \snr \cdot n^{-\frac{2n}{n-d}} \right)\right) + 2\exp \left(-n \log \frac{n}{2}\right) \nonumber \\
&\leq c' \exp \left(-n \log \frac{n}{2}\right). \label{z2}
\end{align}
Combining equations~\eqref{z1} and~\eqref{z2} completes the proof of Lemma~\ref{lem:keylemma}.
\end{proof}

\subsubsection{Proof of Lemma~\ref{lem:tpi}} \label{pf:tpi}

We begin by reducing the problem to the case $x^* = e_1 \|x^*\|_2$. In particular, if $W x^* = e_1 \|x^*\|_2$ for a $d\times d$ unitary matrix $W$ and writing $A = \widetilde{A} W$, we have by rotation invariance of the Gaussian distribution that the the entries of $\widetilde{A}$ are distributed as i.i.d. standard Gaussians. It can be verified that ${T_\Pi = \| I - \Pi \widetilde{A} (\widetilde{A}^\top \widetilde{A})^{-1} (\Pi \widetilde{A})^\top \Pi^* \widetilde{A} e_1 \|_2^2 \|x^* \|_2^2}$. Since $\widetilde{A}\stackrel{d}{=}A$, the reduction is complete.

In order to keep the notation uncluttered, we denote the first column of $A$ by $a$. We also denote the span of the first column of $\Pi A$ by $S_1$ and that of the last $d-1$ columns of $\Pi A$ by $S_{-1}$. Denote their respective orthogonal complements by $S_1^\perp$ and $S^\perp_{-1}$. We then have

\begin{align}
T_\Pi &= \|x^*\|_2^2 \|\Projorthpi a \|_2^2 \nonumber\\
&= \|x^*\|_2^2 \|P_{S^\perp_{-1} \cap S^\perp_1} a \|_2^2 \nonumber\\
&= \|x^*\|_2^2 \|P_{S^\perp_{-1} \cap S^\perp_1} P_{S^\perp_1} a \|_2^2. \nonumber
\end{align}
We now condition on $a$. Consequently, the subspace $S^\perp_1$ is a \emph{fixed} $(n-1)$-dimensional subspace. Additionally, $S^\perp_{-1} \cap S^\perp_1$ is the intersection of a uniformly random $(n-(d-1))$-dimensional subspace with a fixed $(n-1)$-dimensional subspace, and is therefore a uniformly random $(n-d)$-dimensional subspace within $S^\perp_1$. Writing $u = \frac{P_{S^\perp_1} a}{\|P_{S^\perp_1} a\|_2}$, we have
\begin{align}
T_\Pi &\stackrel{d}{=} \|x^*\|_2^2 \|P_{S^\perp_{-1} \cap S^\perp_1}u \|_2^2 \|P_{S^\perp_1} a \|_2^2. \nonumber
\end{align}
Now since $u \in S^\perp_1$, note that $\|P_{S^\perp_{-1} \cap S^\perp_1}u \|_2^2$ is the squared length of a projection of an $(n-1)$-dimensional unit vector onto a uniformly chosen $(n-d)$-dimensional subspace. In other words, denoting a uniformly random projection from $m$ dimensions to $k$ dimensions by $P^m_k$, we have
\begin{align}
\|P_{S^\perp_{-1} \cap S^\perp_1}u \|_2^2 \stackrel{d}{=} \|P^{n-1}_{n-d} v_1\|_2^2 \stackrel{\1}{=} 1 - \|P^{n-1}_{d-1} v_1\|_2^2, \nonumber
\end{align}
where $v_1$ represents a fixed standard basis vector in $n-1$ dimensions. The quantities $P^{n-1}_{n-d}$ and $P^{n-1}_{d-1}$ are projections onto orthogonal subspaces, and step $\1$ is a consequence of the Pythagorean theorem.

Now removing the conditioning on $a$, we see that the term for $d>1$ can be lower bounded by the corresponding $T_\Pi$ for $d= 1$, but scaled by a random factor  -- the norm of a random projection.
Using $T^1_\Pi := \| P^\perp_{S_1} a\|^2_2 \|x^*\|_2^2$ to denote $T_\Pi $ when $d= 1$, we have
\begin{align}
T_\Pi &= (1 - X_{d-1}) T^1_\Pi, \label{eq:twoterms}
\end{align}
where we have introduced the shorthand $X_{d-1} = \|P^{n-1}_{d-1} e_1\|_2^2$.

We first handle the random projection term in equation~\eqref{eq:twoterms} using Lemma~\ref{lem:jl} in Appendix~\ref{sec:tailbound}. In particular, substituting $\beta = (1 - z) \frac{n-1}{d-1}$ in inequality~\eqref{eq:jlbound} yields 
\begin{align}
\Pr\{1 - X_{d-1} \leq z \} &\leq \left( \frac{n-1}{d-1} \right)^{(d-1)/2} \left( \frac{z(n-1)}{n-d} \right)^{(n-d)/2} \nonumber\\
&\stackrel{\1}{\leq} \sqrt{\binom{n-1}{d-1}} \sqrt{\binom{n-1}{n-d}} z^{\frac{n-d}{2}} \nonumber\\
&= \binom{n-1}{d-1} z^{\frac{n-d}{2}} \nonumber\\
&\stackrel{\2}{\leq} 2^{n-1} z^{\frac{n-d}{2}}, \nonumber
\end{align}
where in steps $\1$ and $\2$, we have used the standard inequality ${2^n \geq \binom{n}{r} \geq \left( \frac{n}{r} \right)^r}$. Now setting ${z = n^{\frac{-2n}{n- d}}}$, which ensures that $(1 - z) \frac{n-1}{d-1} >1$ for all $d < n$ and large enough $n$, we have
\begin{align}
\Pr\{1 - X_{d-1} \leq n^{\frac{-2n}{n- d}} \} &\leq \exp \left(- n \log \frac{n}{2}\right). \label{eq:projterm}
\end{align}

\noindent Applying the union bound then yields
\begin{align}
\Pr \{T_\Pi \leq t \|x^*\|_2^2\} \leq \Pr \{ 1 - X_{d-1} \leq n^{\frac{-2n}{n-d}} \} + \Pr \{T^1_\Pi \leq tn^{\frac{2n}{n-d}} \|x^*\|_2^2 \}. \label{nub}
\end{align}
We have already computed an upper bound on $\Pr \{T^1_\Pi \leq tn^{\frac{2n}{n-d}} \|x^*\|_2^2 \}$ in Lemma~\ref{lem:tpi1}. Applying it yields that provided $t \leq h n^{- \frac{2n}{n-d}}$, we have
\begin{align}
\Pr \{T^1_\Pi \leq tn^{\frac{2n}{n-d}} \|x^*\|_2^2 \} \leq 6 \left( \frac{tn^{\frac{2n}{n-d}}}{h} \exp\left(1 - \frac{tn^{\frac{2n}{n-d}}}{h} \right)\right)^{h/10}. \label{eq:term1}
\end{align}

Combining equations~\eqref{eq:term1} and~\eqref{eq:projterm} with the union bound~\eqref{nub} and performing some algebraic manipulation then completes the proof of Lemma~\ref{lem:tpi}.
\qed

\subsection{Information theoretic lower bounds on $\Rpihat$} \label{sec:lbproof}

We prove Theorem~\ref{thm:lowerbound} in Section~\ref{sec:lbp} via the strong converse for the Gaussian channel. We prove the weak converse of Proposition~\ref{rem:lowerbound} in Section~\ref{sec:lb2} by employing Fano's method. We note that the latter is a standard technique for proving minimax bounds in statistical inference problems \cite{yu1997assouad, yang1999information}.

\subsubsection{Proof of Theorem~\ref{thm:lowerbound}} \label{sec:lbp}
We begin by assuming that the design matrix $A$ is fixed, and that the estimator has knowledge of $x^*$ a-priori. Note that the latter cannot make the estimation task any easier. In proving this lower bound, we can also assume that the entries of $Ax^*$ are distinct, since otherwise, perfect permutation recovery is impossible.

Given this setup, we now cast the problem as one of coding over a Gaussian channel. 
Toward this end, consider the codebook 
\begin{align}
\mathcal{C} = \{\Pi A x^* \mid \Pi \in \mathcal{P}_n \}. \nonumber
\end{align}
We may view $\Pi A x^*$ as the codeword corresponding to the permutation $\Pi$, where each permutation is associated to one of $n!$ equally likely messages. Note that each codeword has power $\|Ax^*\|_2^2$.

The codeword is then sent over a Gaussian channel with noise power equal to ${\sum_{i=1}^n \sigma^2 = n\sigma^2}$. The decoding problem is to ascertain from the noisy observations which message was sent, or in other words, to identify the correct permutation.

We now use the non-asymptotic strong converse for the Gaussian channel \cite{yoshihara}. In particular, using Lemma~\ref{lem:strconv} (see Appendix~\ref{shannon}) with $R = \frac{\log n!}{n}$ then yields that for any $\delta'>0$, if
\begin{align}
\frac{\log n!}{n} > \frac{1+\delta'}{2} \log \left(1 + \frac{\|Ax^*\|_2^2}{n\sigma^2}\right), \nonumber
\end{align}
then for any estimator $\Pihat$, we have $\Pr\{ \Pihat \neq \Pi\} \geq 1 - 2 \cdot 2^{-n\delta'}$. For the choice $\delta' = \delta/(2-\delta)$, we have that if
\begin{align}
(2 - \delta) \log \left(\frac{n}{e}\right) > \log \left(1 + \frac{\|Ax^*\|_2^2}{n\sigma^2}\right), \label{condnconv}
\end{align}
then $\Pr\{ \Pihat \neq \Pi\} \geq 1 - 2 \cdot 2^{-n\delta/2}$.
Note that the only randomness assumed so far was in the noise $w$ and the random choice of $\Pi$.

We now specialize the result for the case when $A$ is Gaussian. Toward that end, define the event 
\begin{align}
\mathcal{E}(\delta) = \left\{ 1+\delta \geq \frac{\|Ax^*\|_2^2}{n\|x^*\|_2^2} \right\}. \nonumber
\end{align}

Conditioned on the event $\mathcal{E}(\delta)$, it can be verified that condition~\eqref{lbcondn} implies condition~\eqref{condnconv}. 
We also have
\begin{align}
\Pr\{\mathcal{E}(\delta)\} &= 1 - \Pr \left\{\frac{\|Ax^*\|_2^2}{n\|x^*\|_2^2} > 1 + \delta \right\} \nonumber\\
&\stackrel{\1}{\geq} 1 - c' e^{-c n\delta}, \nonumber
\end{align}
where step $\1$ follows by using the sub-exponential tail bound (see  Lemma~\ref{lem:subexp} in Appendix~\ref{sec:tailbound}), since
$\frac{\|Ax^*\|_2^2}{\|x^*\|_2^2} \sim \chi^2_n$.

Putting together the pieces, we have that provided condition \eqref{lbcondn} holds,
\begin{align}
\Pr\{ \Pihat \neq \Pi^*\} &\geq \Pr\{ \Pihat \neq \Pi^* | \mathcal{E}(\delta)\} \Pr\{\mathcal{E}(\delta)\} \nonumber\\
&= (1 - 2 \cdot 2^{-n\delta/2})(1 - c' e^{-c n\delta}) \nonumber\\
&\geq 1 - c' e^{-cn\delta}. \nonumber
\end{align}
\qed

\noindent We now move on to the proof of Proposition~\ref{rem:lowerbound}.

\subsubsection{Proof of Proposition~\ref{rem:lowerbound}}  \label{sec:lb2}

Proposition~\ref{rem:lowerbound} corresponds to a weak converse in the scenario where the estimator is also provided with the side information that $\Pi^*$ lies within a Hamming ball of radius $2$ around the identity. We carry out the proof for the more general case where $\dH(\Pi^*, I) \leq \bar{h}$ for \emph{any} $\bar{h} \geq 2$, later specializing to the case where $\bar{h} = 2$. We denote such a Hamming ball by $\hballpistar(\bar{h}) := \{ \Pi \in \mathcal{P}_n \;|\; \dH(\Pi, I) \leq \bar{h}\}$.

For the sake of the lower bound, assume that our observation model takes the form
\begin{align}
y = \Pi^* A x^* + \Pi^* w. \label{bothpi}
\end{align} 
Since $w \in \real^n$ is i.i.d. standard normal, model~\eqref{bothpi} has a distribution equivalent to that of model~\eqref{obsmodel}.

Notice that the ML estimation problem is essentially a multi-way hypothesis testing problem among permutations in $\hballpistar(\bar{h})$, and so Fano's method is directly applicable. As before, we assume that the estimator knows $x^*$, and consider a uniformly random choice of $\Pi^* \in \hballpistar(\bar{h})$.

Now note that the observation vector $y$ is drawn from the mixture distribution 
\begin{align}
\mathbb{M}(\bar{h}) = \frac{1}{|\hballpistar(\bar{h})|} \sum_{\Pi \in \mathcal{P}_n} \mathbb{P}_\Pi, \label{mixture}
\end{align} 
where $\mathbb{P}_\Pi$ denotes the Gaussian distribution $\NORMAL(\Pi A x^*, \sigma^2 I_n)$. The following lemma provides a crucial statistic for our bounds.

\begin{lemma} \label{lem:det}
For $y$ drawn according to the mixture distribution $\mathbb{M}(\bar{h})$, we have 
\begin{align}
\det \EE \left[yy^\top\right] &\leq (\sigma^2  + \|x^*\|_2^2)^n \left( 1 + n\right) \left( \frac{\bar{h}}{n} \right)^{n-1}. \label{eq:var}
\end{align}
\end{lemma}

We prove Lemma~\ref{lem:det} in Section~\ref{proof:det}. Taking it as given, we are now in a position to prove Proposition \ref{rem:lowerbound}.

\begin{proof}[Proof of Proposition \ref{rem:lowerbound}]
We begin by using Fano's inequality to bound the probability of error of any tester random ensemble of $\Pi^*$. 
In particular, for any estimator $\Pihat$ that is a measurable function of the pair $(y, A)$, we have

\begin{align}
\Pr\{\Pihat \neq \Pi^* \} \geq 1 - \frac{I(\Pi^*; y, A) + \log 2}{\log |\hballpistar(\bar{h})|}. \label{eq:fano}
\end{align}

Applying the chain rule for mutual information yields
\begin{align}
I(\Pi^*; y, A) &= I(\Pi^*; y | A) + I(\Pi^*, A) \nonumber \\
&\stackrel{\1}{=} I(\Pi^*; y | A) \nonumber \\
&= \EE_A \left[ I(\Pi^*; y | A = \alpha) \right],
\end{align}
where step $\1$ follows since $\Pi^*$ is chosen independently of $A$. We now evaluate the mutual information term $I(\Pi^*; y|A = \alpha)$, which we denote by $I_\alpha (\Pi^*; y)$. 
Letting ${H_\alpha(y) := H(y | A=\alpha)}$ denote the conditional entropy of $y$ given a fixed realization of $A$, we have
\begin{align}
I_\alpha (\Pi^*; y) &= H_\alpha (y) - H_\alpha (y | \Pi^*) \nonumber \\
&\stackrel{\2}{\leq} \frac{1}{2} \log \det \cov yy^\top - \frac{n}{2}\log \sigma^2, \nonumber
\end{align}
where the covariance is evaluated with $A= \alpha$, and in step $\2$, we have used two key facts:

\noindent (a) Gaussians maximize entropy for a fixed covariance, which bounds the first term, and 

\noindent (b) For a fixed realization of $\Pi^*$, the vector $y$ is composed of $n$ uncorrelated Gaussians. This leads to an explicit evaluation of the second term.

Now taking expectations over $A$ and noting that $\cov yy^\top \preceq \EE_w \left[ yy^\top\right]$, we have from the concavity of the log determinant function and Jensen's inequality that
\begin{align}
I(\Pi^* ; y | A) &= \EE_A \left[ I_\alpha (\Pi^*; y) \right] \nonumber \\
&\leq \frac{1}{2} \log\det \EE \left[ yy^\top \right] - \frac{n}{2}\log \sigma^2, \label{eq:mi}
\end{align}
where the expectation in the last line is now taken over randomness in both $A$ and $w$.


Applying Lemma \ref{lem:det}, we can then substitute inequality~\eqref{eq:var} into bound~\eqref{eq:mi}, which yields
\begin{align}
\EE_A \left[ I_\alpha (\Pi^*; y) \right] \leq \frac{n}{2} \log \left( 1 + \SNR \right) + \frac{n-1}{2} \log \frac{\bar{h}}{n} + \frac{1}{2}\log \left( 1 + n\right). \nonumber
\end{align}
Finally, substituting into the Fano bound~\eqref{eq:fano} yields the lower bound
\begin{align}
\Pr\{\Pihat \neq \Pi^*\} &\geq 1 - \frac{\frac{n}{2} \log \left( 1 + \SNR \right) + \frac{n-1}{2} \log \frac{\bar{h}}{n} + \frac{1}{2} \log \left( 1 + n\right) + \log 2}{\log |\hballpistar(\bar{h})|} \nonumber\\
&\stackrel{\2}{\geq} 1 - \frac{\frac{n}{2} \log \left( 1 + \SNR \right) + \frac{n-1}{2} \log \frac{\bar{h}}{n} + \frac{1}{2} \log \left( 1 + n\right) + \log 2}{\bar{h}\log (n/e)}, \nonumber
\end{align}
where in step $\2$, we have used the fact that 
\begin{align}
|\hballpistar(\bar{h})| = \binom{n}{\bar{h}}\cdot \bar{h}! 
\geq \left( n / \bar{h} \right)^{\bar{h}} \left(\frac{\bar{h}}{e}\right)^{\bar{h}}. \nonumber
\end{align}
In other words, if
\begin{align}
\log \left[4\left( 1 + \SNR \right)\right] \leq \frac{\bar{h}}{n} \log (n/e) + \frac{n-1}{n}\log (n/ \bar{h})  - \frac{\log \left(1+n\right)}{n}, \label{condnlb2}
\end{align}
then $\Rpihat \geq 1/2$ for any estimator $\Pihat$.
Evaluating condition~\eqref{condnlb2} for $\bar{h} = 2$ yields the required result. 
\end{proof}

\subsubsection{Proof of Lemma~\ref{lem:det}} \label{proof:det}
We first explicitly calculate the matrix $Y := \EE \left[yy^\top\right]$. Note that the diagonal entries take the form 
\begin{align}
Y_{ii} = (x^*)^\top \EE \left[a_{\pi_i} a_{\pi_i}^\top\right] x^* + \EE \left[w^2_{\pi_i} \right] 
= \|x^*\|_2^2 + \sigma^2. \nonumber
\end{align}
The off-diagonal entries can be evaluated as
\begin{align}
Y_{ij} = (x^*)^\top \EE \left[ a_{\pi_i} a_{\pi_j}^\top\right] x^* + \EE \left[w_{\pi_i} w_{\pi_j}\right]
\stackrel{\1}{=} \left(\frac{n-\bar{h}}{n} + \frac{\bar{h}}{n^2}\right) \left(\|x^*\|_2^2 + \sigma^2\right), \text{ for } i\neq j, \nonumber
\end{align}
where step $\1$ follows since 
\begin{align}
\pi_i = \pi_j \text{ with probability } \frac{n-\bar{h}}{n} + \frac{\bar{h}}{n^2}. \label{hprob}
\end{align}
Equation~\eqref{hprob} is a consequence of the fact that a uniform permutation over $\hballpistar(\bar{h})$ can be generated by first picking $\bar{h}$ positions (the permutation set) out of $[n]$ uniformly at random, and then uniformly permuting those $\bar{h}$ positions. The probability that $\pi_i = \pi_j$ is equal to the probability that $\pi_i = i$, an event that occurs if:

\noindent (a) position $i$ is not chosen in the permutation set, which happens with probability $\frac{n-\bar{h}}{n}$, or if 

\noindent (b) position $i$ is in the permutation set but the permutation maps $i$ to itself, which happens with probability $\frac{\bar{h}}{n} \frac{1}{n}$.

Hence, the determinant of $Y$ is given by $\det Y = (\|x^*\|_2^2 + \sigma^2)^n \det \widebar{Y}$, where we have defined $\widebar{Y} := \frac{1}{\|x^*\|_2^2 + \sigma^2} Y$. Note that $\widebar{Y}$ is a highly structured matrix, and so its determinant can be computed exactly. In particular, letting the scalar $\beta = 1$ denote the identical diagonal entries of $\widebar{Y}$ and the scalar $\gamma$ denote its identical off-diagonal entries, 
it is easy to verify that the all ones vector ${\bf 1}$ is an eigenvector of $\widebar{Y}$, with corresponding eigenvalue $\beta + (n-1)\gamma$. Additionally, for any vector $v$ that obeys ${\bf 1}^\top v  = 0$, we have
\begin{align}
\widebar{Y} v = (\beta - \gamma) v + \gamma ( v^\top {\bf 1}) {\bf 1} = (\beta - \gamma) v, \nonumber
\end{align}
and so the remaining $n-1$ eigenvalues are identically $\beta- \gamma$.

Substituting for $\beta$ and $\gamma$, the eigenvalues of $\widebar{Y}$ are given by 
$$\lambda_1(\widebar{Y}) = 1 + \frac{(n - \bar{h})(n-1)}{n} + \frac{\bar{h}(n-1)}{n^2}, \text{ and}$$
$$\lambda_2(\widebar{Y}) = \lambda_3(\widebar{Y}) = \cdots = \lambda_n(\widebar{Y}) = \frac{\bar{h}}{n} - \frac{\bar{h}}{n^2}.$$


Hence, we have
\begin{align}
\det \widebar{Y} &= \left(1 + \frac{(n - \bar{h})(n-1)}{n} + \frac{\bar{h}(n-1)}{n^2}\right) \left( \frac{\bar{h}}{n} - \frac{\bar{h}}{n^2}\right)^{n-1} \nonumber\\
&\leq \left(1 + n - \bar{h} + \frac{\bar{h}}{n}\right) \left( \frac{\bar{h}}{n}\right)^{n-1} \nonumber\\
&\leq \left(1 + n\right) \left( \frac{\bar{h}}{n}\right)^{n-1}, \nonumber
\end{align}
where in the last step, we have used the fact that $0 < \bar{h} \leq n$. This completes the proof. \qed

\subsection{Proof of Theorem \ref{thm:approx}} \label{sec:approx}
We now prove Theorem \ref{thm:approx} for approximate permutation recovery. For any estimator $\Pihat$, we denote by the indicator random variable $E(\Pihat, D)$ whether or not the $\Pihat$ has acceptable distortion, i.e., $E(\Pihat, D) = \mathbb{I}[ \dH(\Pihat, \Pi^*) \geq D]$, with $E= 1$ representing the error event. For $\Pi^*$ picked uniformly at random in $\mathcal{P}_n$, we have the following variant of Fano's inequality.
\begin{lemma} \label{fano2}
The probability of error is lower bounded as
\begin{align}
\Pr\{E(\Pihat, D) = 1\} \geq 1 - \frac{I(\Pi^*; y, A) + \log 2}{\log n! - \log \frac{n!}{(n-D+1)!}}. \label{fano1}
\end{align}
\end{lemma}
Taking the lemma as given for the moment, we are now ready to prove Theorem \ref{thm:approx}.

\begin{proof}[Proof of Theorem \ref{thm:approx}] 
The proof of the theorem follows by upper bounding the mutual information term. In particular, we have
\begin{align}
I(\Pi^*; y, A) &\leq \EE_A \left[ I_a (\Pi^*; y) \right] \nonumber \\
&\leq \frac{1}{2} \log\det \EE \left[ yy^\top \right] - \frac{n}{2}\log \sigma^2  \nonumber \\
&\stackrel{\1}{\leq} \frac{n}{2} \log \left( 1+ \SNR\right), \nonumber
\end{align}
where the expectation on the RHS is taken over both $\Pi^*$ and $A$.
Also, step $\1$ follows from the AM-GM inequality for PSD matrices $\det X \leq \left(\frac{1}{n}\trace X \right)^n$, and by noting that the diagonal entries of the matrix $\EE \left[ yy^\top \right]$ are all equal to ${\|x^*\|_2^2 + \sigma^2}$.

Combining the pieces, we now have that ${\Pr\{\Pihat \neq \Pi^*\} \geq 1/2}$ if
\begin{align}
n \log \left(1 + \SNR \right) \leq (n - D+1) \log \left(\frac{n-D+1}{2e} \right), \label{distcondn}
\end{align}
which completes the proof.
\end{proof}
\subsubsection{Proof of Lemma \ref{fano2}}
We use the shorthand $E:= E(\Pihat, D)$ in this proof to simplify notation. Proceeding by the usual proof of Fano's inequality, we begin by expanding $H(E, \Pi^* | y, A = a, \Pihat)$ in two ways:
\begin{subequations} \label{eq:ent}
\begin{align}
H(E, \Pi^* | y, A, \Pihat) &= H(\Pi^* | y, A, \Pihat) + H(E | \Pi^*, y, A, \Pihat) \\
&= H(E | y, A, \Pihat) + H(\Pi^* | E, y, A, \Pihat).
\end{align}
\end{subequations}
Since $\Pi^* \rightarrow (y, A) \rightarrow \Pihat$ forms a Markov chain, we have $H(\Pi^* | y, A, \Pihat) = H(\Pi^* | y, A)$. Non-negativity of entropy yields $H(E|\Pi^*, y, A, \Pihat) \geq 0$. Since conditioning cannot increase entropy, we have
$H(E | y, A, \Pihat) \leq H(E) \leq \log 2$, and $H(\Pi^* | E, y, A, \Pihat) \leq H(\Pi^* | E, \Pihat)$. Combining all of this with equations \eqref{eq:ent} yields
\begin{align}
H(\Pi^* | y, A) &\leq H(\Pi^* | E, \Pihat) + \log 2 \nonumber \\
&= \Pr\{E = 1\} H(\Pi^* | E = 1, \Pihat) + \left(1 - \Pr\{E = 1\} \right) H(\Pi^* | E = 0, \Pihat) + \log 2. \label{ent2}
\end{align}
We now use the fact that uniform distributions maximize entropy to bound the two terms as $H(\Pi^* | E = 1, \Pihat) \leq H(\Pi^*) = \log n!$, and $H(\Pi^* | E = 0, \Pihat) \leq \log \frac{n!}{(n-D+1)!}$, where the last inequality follows since $E = 0$ reveals that $\Pi^*$ is within a Hamming ball of radius $D - 1$ around $\Pihat$, and the cardinality of that Hamming ball is $\frac{n!}{(n-D+1)!}$.

Substituting back into inequality \eqref{ent2} yields
\begin{align}
&\Pr\{E = 1\} \left( \log n! - \log \frac{n!}{(n-D+1)!} \right) + H(\Pi^*) \nonumber \\
& \qquad \geq H(\Pi^* | y, A) - \log 2 - \log \frac{n!}{(n-D+1)!} + \log n!, \nonumber
\end{align}
where we have added the term $H(\Pi^*) = \log n!$ to both sides. Simplifying then yields inequality \eqref{fano1}.
\qed

\subsection{Proofs of computational aspects}
In this section, we prove Theorem \ref{thm:1D} by providing an efficient algorithm for the $d = 1$ case and showing NP-hardness for the $d>1$ case.

\subsubsection{Proof of Theorem~\ref{thm:1D}: $d = 1$ case} \label{sec:1Dproof}
In order to prove the theorem, we need to show an algorithm that performs the optimization~\eqref{eq:MLperm} efficiently. Accordingly, note that for the case when $d = 1$, equation~\eqref{eq:MLperm} can be rewritten as
\begin{align}
\Pihatml &= \arg\max_{\Pi} \| a_{\Pi}^\top y\|^2  \nonumber\\
&= \arg\max_{\Pi} \max\left\{ a_{\Pi}^\top y, -a_{\Pi}^\top y \right\}  \nonumber\\
&= \arg\min_{\Pi} \max\Big\{ \| a_{\Pi} -  y \|_2^2, \|a_{\Pi} + y \|_2^2 \Big\}, \label{eq:mindistance}
\end{align}
where the last step follows since $2 a_{\Pi}^\top y = \|a \|^2 + \|y\|^2 - \| a_{\Pi} -  y \|_2^2$, and the first two terms do not involve optimizing over $\Pi$.

Once the optimization problem has been written in the form~\eqref{eq:mindistance}, it is easy to see that it can be solved in polynomial time. In particular, using the fact that for fixed vectors $p$ and $q$, $\|p_{\Pi} - q\|$ is minimized for $\Pi$ that sorts $a$ according to the order of $b$, we see that Algorithm~\ref{algo:determ} computes $\Pihatml$ exactly.

\begin{algorithm}
\KwIn{design matrix (vector) $a$, observation vector $y$}
$\Pi_1 \gets$ permutation that sorts $a$ according to $y$ \\
$\Pi_2 \gets$ permutation that sorts $-a$ according to $y$ \\
$\Pihatml \gets  \arg\max \{|a_{\Pi_1}^\top y|, |a_{\Pi_2}^\top y| \}$ \\
\KwOut{\Pihatml}
\caption{Exact algorithm for implementing equation~\eqref{eq:MLperm} for the case when $d=1$.} 
\label{algo:determ}
\end{algorithm}

The procedure defined by Algorithm~\ref{algo:determ} is clearly the correct thing to do in the noiseless case: in this case, $x^*$ is a scalar value that scales the entries of $a$, and so the correct permutation can be identified by a simple sorting operation. Two such operations suffice, one to account for when $x^*$ is positive and one more for when it is negative. Since each sort operation takes $\BigO(n\log n)$ steps, Algorithm~\ref{algo:determ} can be executed in nearly linear time.
\qed

\subsubsection{Proof of Theorem~\ref{thm:1D}: NP-hardness} \label{sec:nphardness}
In this section, we show that given a vector matrix pair $(y,A) \in \real^n \times \real^{n\times d}$, it is NP-hard to determine whether the equation $y = \Pi A x$ has a solution for a permutation matrix $\Pi \in \mathcal{P}_n$ and vector $x \in \real^d$.
Clearly, this is sufficient to show that the problem~\eqref{eq:ML} is NP-hard to solve in the case when $A$ and $y$ are arbitrary.


Our proof involves a reduction from the \textsf{PARTITION} problem, the decision version of which is defined as the following.
\begin{definition}[\textsf{PARTITION}]
Given $d$ integers $b_1, b_2, \cdots, b_d$, does there exist a subset $S \subset [d]$ such that $$\sum_{i \in S} b_i = \sum_{i \in [d]\setminus S} b_i?$$
\end{definition} 

It is well known \cite{papadimitriou1998combinatorial} that
\textsf{PARTITION} is NP-complete.
Also note that asking whether or not equation~\eqref{obsmodel} has a solution $(\Pi, x)$ is equivalent to determining whether or not there exists a permutation $\pi$ and a vector $x$ such that $y_\pi = Ax$ has a solution. We are now ready to prove the theorem.

Given an instance $b_1, \cdots b_d$ of \textsf{PARTITION}, define a vector $y \in \mathbb{Z}^{2d+1}$ with entries
\begin{align}
y_i :=
\begin{cases}
b_i, \text{ if } i \in [d] \\
0, \text{ otherwise.} \nonumber
\end{cases}
\end{align}
Also define the $2d+1 \times 2d$ matrix
\begin{align}
A := 
\begin{bmatrix}
\multicolumn{2}{c}{ I_{2d}} \\ 
1_d^\top &  -1_d^\top
\end{bmatrix}. \nonumber
\end{align}
Clearly, the pair $(y, A)$ can be constructed from $b_1, \cdots b_d$ in time polynomial in $n = 2d + 1$. We now claim that $y_\pi = A x$ has a solution $(\Pi, x)$ if and only if there exists a subset $S \subset [d]$ such that $\sum_{i \in S} b_i = \sum_{i \in [d]\setminus S} b_i$. 

By converting to row echelon form, we see that $y_\pi = A x$ if and only if
\begin{align}
\sum_{i \mid \pi(i) \leq d} y_{\pi(i)} = \sum_{i \mid \pi(i) > d} y_{\pi(i)}, \label{eq:np}
\end{align}
and equation~\eqref{eq:np} holds, by construction, if and only if for $S = \{i \mid \pi(i) \leq d\} \cap [d]$, we have 
$$\sum_{i \in S} b_i = \sum_{i \in [d]\setminus S} b_i.$$
This completes the proof. 
\qed

\section{Discussion}
We analyzed the problem of exact permutation recovery in the linear regression model, and provided necessary and sufficient conditions that are tight in most regimes of $n$ and $d$. We also provided a converse for the problem of approximate permutation recovery to within some Hamming distortion. It is still an open problem to characterize the fundamental limits of exact and approximate permutation recovery for all regimes of $n$, $d$ and the allowable distortion $D$. In the context of exact permutation recovery, we believe that the limit suggested by Theorem~\ref{thm:upperbound} is tight for all regimes of $n$ and $d$, but showing this will likely require a different technique. In particular, as pointed out in Remark~\ref{rem:sideinfo}, all of our lower bounds assume that the estimator is provided with $x^*$ as side information; it is an interesting question as to whether stronger lower bounds can be obtained without this side information.

On the computational front, many open questions remain. The primary question concerns the design of computationally efficient estimators that succeed in similar SNR regimes. We have already shown that the maximum likelihood estimator, while being statistically optimal for moderate $d$, is computationally hard to compute in the worst case. Showing a corresponding hardness result for random $A$ is also an open problem.
Finally, while this paper mainly addresses the problem of permutation recovery, the complementary problem of recovering $x^*$ is also interesting, and we plan to investigate its fundamental limits in future work.
\subsection*{Acknowledgements}
This work was partially supported by NSF Grants CCF-1528132 and CCF-0939370 (Center for Science of Information), Office of Naval Research MURI
grant DOD-002888, Air Force Office of Scientific Research Grant
AFOSR-FA9550-14-1-001, Office of Naval Research grant ONR-N00014, as
well as National Science Foundation Grant CIF-31712-23800.


\appendix

\section{Auxiliary results}
In this section, we prove a preliminary lemma about permutations that is useful in many of our proofs. We also derive tight bounds on the lower tails of $\chi^2$-random variables and state an existing result on tail bounds for random projections.

\subsection{Independent sets of permutations} \label{sec:perm}
In this section, we prove a combinatorial lemma about permutations, for which we need to set up some additional notation. For a permutation $\pi$ on $k$ objects, let $G_\pi$ denote the corresponding undirected incidence graph, i.e., $V(G_\pi) = [k]$, and $(i, j) \in E(G_\pi)$ iff $j = \pi(i)$ or $i = \pi(j)$.
\begin{lemma}\label{lem:perm}
Let $\pi$ be a permutation on $k \geq 3$ objects such that $\dH(\pi, I) = k$. 
Then the vertices of $G_\pi$ can be partitioned into three sets $V_1$, $V_2$, $V_3$ such that each is an independent set, and $|V_1|, |V_2|, |V_3| \geq \lfloor\frac{k}{3} \rfloor \geq \frac{k}{5}$.
\end{lemma}

\begin{proof}
Note that for any permutation $\pi$, the corresponding graph $G_\pi$ is composed of cycles, and the vertices in each cycle together form an independent set. Consider one such cycle. We can go through the vertices in the order induced by the cycle, and alternate placing them in each of the $3$ partitions. Clearly, this produces independent sets, and furthermore, having $3$ partitions ensures that the last vertex in the cycle has some partition with which it does not share edges. If the cycle length $C \equiv 0 \pmod 3$, then each partition gets $C/3$ vertices, otherwise the smallest partition has $\lfloor C/3 \rfloor$ vertices. The partitions generated from the different cycles can then be combined (with relabelling, if required) to ensure that the largest partition has cardinality at most $1$ more than that of the smallest partition.
\end{proof}

\subsection{Tails bounds on $\chi^2$ random variables and random projections} \label{sec:tailbound}
In our analysis, we require tight control on lower tails of $\chi^2$ random variables. The following lemma provides one such bound.

\begin{lemma} \label{lem:tailbound}
Let $Z_\ell$ denote a $\chi^2$ random variable with $\ell$ degrees of freedom. Then for all $p \in [0, \ell]$, we have
\begin{align}
\Pr \{Z_\ell \leq p\} \leq \left( \frac{p}{\ell} \exp\left(1 - \frac{p}{\ell} \right)\right)^{\ell/2} = \exp\left( - \frac{\ell}{2} \left[ \log \frac{\ell}{p} + \frac{p}{\ell} - 1 \right] \right)\label{eq:lowertail}
\end{align}
\end{lemma}
\begin{proof}
The lemma is a simple consequence of the Chernoff bound. In particular, we have for all $\lambda > 0$ that
\begin{align}
\Pr \{ Z_\ell \leq p \} &= \Pr \{ \exp(-\lambda Z_\ell) \geq \exp(-\lambda p)\}  \nonumber\\
&\leq \exp(\lambda p) \EE \left[ \exp(-\lambda Z_\ell) \right] \nonumber\\
&= \exp(\lambda p) (1 + 2\lambda)^{-\frac{\ell}{2}}. \label{eq:lasteq}
\end{align}
where in the last step, we have used $\EE \left[ \exp(-\lambda Z_\ell) \right] = (1 + 2\lambda)^{-\frac{\ell}{2}}$, which is valid for all $\lambda > -1/2$. Minimizing the last expression over $\lambda > 0$ then yields the choice $\lambda^* = \frac{1}{2} \left( \frac{\ell}{p} - 1 \right)$, which is greater than $0$ for all $0 \leq p \leq \ell$. Substituting this choice back into equation~\eqref{eq:lasteq} proves the lemma.
\end{proof}

We also state the following lemma for general sub-exponential random variables (see, for example, Wainwright~\cite[Theorem 2.2]{wainwrig2015high}). We use it in the context of $\chi^2$ random variables.
\begin{lemma} \label{lem:subexp}
Let $X$ be a sub-exponential random variable. Then for all $t> 0$, we have
\begin{align}
\Pr\{|X - \EE [X] | \geq t\} \leq c' e^{-c t}. \nonumber
\end{align}
\end{lemma}

Lastly, we require tail bounds on the norms of random projections, a problem that has been studied extensively in the literature on dimensionality reduction. The following lemma, a consequence of the Chernoff bound, is taken from Dasgupta and Gupta \cite[Lemma 2.2b]{dasguptajl}.
\begin{lemma}[\cite{dasguptajl}] \label{lem:jl}
Let $x$ be a fixed $n$-dimensional vector, and let $P^n_d$ be a projection matrix from $n$-dimensional space to a uniformly randomly chosen $d$-dimensional subspace. Then we have for every $\beta > 1$ that
\begin{align}
\Pr \{ \| P^n_d x\|_2^2 \geq \frac{\beta d}{n} \|x\|_2^2 \} \leq \beta^{d/2} \left(1 + \frac{(1 - \beta)d}{n -d} \right)^{(n - d)/2}. \label{eq:jlbound}
\end{align}
\end{lemma} 

\subsection{Strong converse for Gaussian channel capacity} \label{shannon}
The following result due to Shannon \cite{shannon1959probability} provides a strong converse for the Gaussian channel. The non-asymptotic version as stated here was also derived by Yoshihara \cite{yoshihara}.
\begin{lemma}[\cite{yoshihara}] \label{lem:strconv}
Consider a vector Gaussian channel on $n$ coordinates with message power $P$ and noise power $\sigma^2$, whose capacity is given by $\widebar{R} = \log\left(1 + \frac{P}{\sigma^2} \right)$. For any codebook $\mathcal{C}$ with $|\mathcal{C}| = 2^{nR}$, if for some $\epsilon > 0$ we have
\begin{align}
R > (1+\epsilon) \widebar{R}, \nonumber
\end{align}
then the probability of error $p_e \geq 1 - 2 \cdot 2^{-n\epsilon}$ for $n$ large enough.
\end{lemma}

\section{Proof of Lemma~\ref{lem:noise}} \label{app:noiselemma}
We prove each claim of the lemma separately.
\subsection{Proof of claim~\eqref{noiseterm}}
To start, note that by definition of the linear model, we have $\| \Projorthpistar y\|_2^2 = \| \Projorthpistar w\|_2^2$. Letting $Z_\ell$ denote a $\chi^2$ random variable with $\ell$ degrees of freedom, we claim that
$$\| \Projorthpistar w\|_2^2 - \| \Projorthpi w\|_2^2 = Z_k - \tilde{Z}_k,$$ where $k := \min(d, \dH(\Pi, \Pi^*))$.

For the rest of the proof, we adopt the shorthand $\Pi \setminus \Pi' := {\sf range}(\Pi A) \setminus {\sf range}(\Pi' A)$, and $\Pi \cap \Pi' := {\sf range}(\Pi A) \cap {\sf range}(\Pi' A)$. Now, by the Pythagorean theorem, we have
$$\| \Projorthpistar w\|_2^2 - \| \Projorthpi w\|_2^2 = \| P_{\Pi} w\|_2^2 - \| P_{\Pi^*} w\|_2^2.$$ Splitting it up further, we can then write
$$\| P_{\Pi} w\|_2^2 = \| P_{\Pi \cap \Pi^*} w\|_2^2 + \| (P_{\Pi} - P_{\Pi \cap \Pi^*}) w\|_2^2,$$ where we have used the fact that $P_{\Pi \cap \Pi^*} P_{\Pi} = P_{\Pi \cap \Pi^*} = P_{\Pi \cap \Pi^*} P_{\Pi^*}$.

Similarly for the second term, we have $\| P_{\Pi^*} w\|_2^2 = \| P_{\Pi \cap \Pi^*} w\|_2^2 + \| (P_{\Pi^*} - P_{\Pi \cap \Pi^*}) w\|_2^2,$ and hence,

$$\| P_{\Pi} w\|_2^2 - \| P_{\Pi^*} w\|_2^2 = \| (P_{\Pi} - P_{\Pi \cap \Pi^*}) w\|_2^2 - \| (P_{\Pi^*} - P_{\Pi \cap \Pi^*}) w\|_2^2.$$

Now each of the two projection matrices above has rank\footnote{With probability 1} ${\sf dim} (\Pi \setminus \Pi^*) = k$, which completes the proof of the claim. 
To prove the lemma, note that for any $\delta > 0$, we can write $$\Pr\{ \Fspace_1(\delta) \} \leq \Pr\{|Z_k - k | \geq \delta/2 \} + \Pr\{|\tilde{Z}_k - k | \geq \delta/2 \}.$$ Using the sub-exponential tail-bound on $\chi^2$ random variables (see Lemma~\ref{lem:subexp} in Appendix~\ref{sec:tailbound}) completes the proof.

\subsection{Proof of claim~\eqref{signalterm}}
We begin by writing 
\begin{align}
\Pr\nolimits_w\{ \Fspace_2(\delta) \} = \Pr\nolimits_w\left \{ \underbrace{\|\Projorthpi \Pi^* A x^* \|^2_2  + 2 \langle \Projorthpi \Pi^* A x^*, \Projorthpi w\rangle}_{R(A, w)} \leq 2\delta \right\}. \nonumber
\end{align} 
We see that conditioned on $A$, the random variable $R(A, w)$ is distributed as $\mathcal{N}(T_\Pi, 4\sigma^2 T_\Pi)$, where we have used the shorthand $T_\Pi := \|\Projorthpi \Pi^* A x^* \|^2_2$.

So applying standard Gaussian tail bounds (see, for example, Wainwright~\cite[Example~2.1]{wainwrig2015high}), we have $$\Pr\nolimits_w \{\Fspace_2(\delta)\} \leq \exp \left( - \frac{(T_\Pi - 2\delta)^2}{8\sigma^2 T_\Pi}\right).$$ Setting $\delta = \delta^* := \frac{1}{3}T_\Pi$ completes the proof.

\bibliographystyle{alpha}
\bibliography{research}

\newcommand{\etalchar}[1]{$^{#1}$}
\begin{thebibliography}{LdABN{\etalchar{+}}07}

\bibitem[Bal62]{balakrishnan1962problem}
A.~V. Balakrishnan.
\newblock On the problem of time jitter in sampling.
\newblock {\em IRE Transactions on Information Theory}, 8(3):226--236, 1962.

\bibitem[Blu11]{blumensath2011sampling}
T.~Blumensath.
\newblock Sampling and reconstructing signals from a union of linear subspaces.
\newblock {\em IEEE Transactions on Information Theory}, 57(7):4660--4671,
  2011.

\bibitem[CD16]{collier2016minimax}
O.~Collier and A.~S. Dalalyan.
\newblock Minimax rates in permutation estimation for feature matching.
\newblock {\em Journal of Machine Learning Research}, 17(6):1--31, 2016.

\bibitem[Cha15]{chatterjee}
S.~Chatterjee.
\newblock Matrix estimation by universal singular value thresholding.
\newblock {\em The Annals of Statistics}, 43(1):177--214, 2015.

\bibitem[CT06]{candesCS}
E.~J. Candes and T.~Tao.
\newblock Near-optimal signal recovery from random projections: Universal
  encoding strategies?
\newblock {\em Information Theory, IEEE Transactions on}, 52(12):5406--5425,
  2006.

\bibitem[DDDS04]{david2004softposit}
P.~David, D.~Dementhon, R.~Duraiswami, and H.~Samet.
\newblock Softposit: Simultaneous pose and correspondence determination.
\newblock {\em International Journal of Computer Vision}, 59(3):259--284, 2004.

\bibitem[DG03]{dasguptajl}
S.~Dasgupta and A.~Gupta.
\newblock An elementary proof of a theorem of {J}ohnson and {L}indenstrauss.
\newblock {\em Random Structures \& Algorithms}, 22(1):60--65, 2003.

\bibitem[EBDG14]{emiya2014compressed}
V.~Emiya, A.~Bonnefoy, L.~Daudet, and R.~Gribonval.
\newblock Compressed sensing with unknown sensor permutation.
\newblock In {\em Acoustics, Speech and Signal Processing (ICASSP), 2014 IEEE
  International Conference on}, pages 1040--1044. IEEE, 2014.

\bibitem[FJBd13]{fogel2013convex}
F.~Fogel, R.~Jenatton, F.~Bach, and A.~d'Aspremont.
\newblock Convex relaxations for permutation problems.
\newblock In {\em Advances in Neural Information Processing Systems}, pages
  1016--1024, 2013.

\bibitem[FMR16]{rigollet}
N.~Flammarion, C.~Mao, and P.~Rigollet.
\newblock Optimal rates of statistical seriation.
\newblock {\em arXiv preprint arXiv:1607.02435}, 2016.

\bibitem[HGG09]{huang2009fourier}
J.~Huang, C.~Guestrin, and L.~Guibas.
\newblock Fourier theoretic probabilistic inference over permutations.
\newblock {\em The Journal of Machine Learning Research}, 10:997--1070, 2009.

\bibitem[HM99]{huang1999cap3}
X.~Huang and A.~Madan.
\newblock {CAP3: A DNA} sequence assembly program.
\newblock {\em Genome Research}, 9(9):868--877, 1999.

\bibitem[KSF{\etalchar{+}}09]{keller2009identity}
L.~Keller, M.~J. Siavoshani, C.~Fragouli, K.~Argyraki, and S.~Diggavi.
\newblock Identity aware sensor networks.
\newblock In {\em INFOCOM 2009, IEEE}, pages 2177--2185. IEEE, 2009.

\bibitem[LD08]{lu2008theory}
Y.~M. Lu and M.~N. Do.
\newblock A theory for sampling signals from a union of subspaces.
\newblock {\em IEEE Transactions on Signal Processing}, 56(6):2334--2345, 2008.

\bibitem[LdABN{\etalchar{+}}07]{loiola2007survey}
E.~M. Loiola, N.~M.~M. de~Abreu, P.~O. Boaventura-Netto, P.~Hahn, and
  T.~Querido.
\newblock A survey for the quadratic assignment problem.
\newblock {\em European Journal of Operational Research}, 176(2):657--690,
  2007.

\bibitem[LR14]{little2014statistical}
R.~J.~A. Little and D.~B. Rubin.
\newblock {\em Statistical analysis with missing data}.
\newblock John Wiley \& Sons, 2014.

\bibitem[LW12]{loh2012corrupted}
P.~Loh and M.~J. Wainwright.
\newblock Corrupted and missing predictors: Minimax bounds for high-dimensional
  linear regression.
\newblock In {\em Information Theory Proceedings (ISIT), 2012 IEEE
  International Symposium on}, pages 2601--2605. IEEE, 2012.

\bibitem[MSC09]{marques2009subspace}
M.~Marques, M.~Sto{\v{s}}i{\'c}, and J.~Costeira.
\newblock Subspace matching: Unique solution to point matching with geometric
  constraints.
\newblock In {\em Computer Vision, IEEE 12th International Conference on},
  pages 1288--1294. IEEE, 2009.

\bibitem[NS08]{narayanan2008robust}
A.~Narayanan and V.~Shmatikov.
\newblock Robust de-anonymization of large sparse datasets.
\newblock In {\em Security and Privacy, 2008. SP 2008. IEEE Symposium on},
  pages 111--125. IEEE, 2008.

\bibitem[PG06]{poore2006some}
A.~B. Poore and S.~Gadaleta.
\newblock Some assignment problems arising from multiple target tracking.
\newblock {\em Mathematical and Computer Modelling}, 43(9):1074--1091, 2006.

\bibitem[PS98]{papadimitriou1998combinatorial}
C.~H. Papadimitriou and K.~Steiglitz.
\newblock {\em Combinatorial optimization: algorithms and complexity}.
\newblock Courier Corporation, 1998.

\bibitem[RMS12]{rose1}
C.~Rose, I.~S. Mian, and R.~Song.
\newblock Timing channels with multiple identical quanta.
\newblock {\em arXiv preprint arXiv:1208.1070}, 2012.

\bibitem[Rob51]{robinson1951method}
W.~S. Robinson.
\newblock A method for chronologically ordering archaeological deposits.
\newblock {\em American Antiquity}, pages 293--301, 1951.

\bibitem[SBGW15]{nihar}
N.~B. Shah, S.~Balakrishnan, A.~Guntuboyina, and M.~J. Wainright.
\newblock Stochastically transitive models for pairwise comparisons:
  Statistical and computational issues.
\newblock {\em arXiv preprint arXiv:1510.05610}, 2015.

\bibitem[Sha59]{shannon1959probability}
C.~E. Shannon.
\newblock Probability of error for optimal codes in a {G}aussian channel.
\newblock {\em Bell System Technical Journal}, 38(3):611--656, 1959.

\bibitem[SZ99]{schulman1999asymptotically}
L.~J. Schulman and D.~Zuckerman.
\newblock Asymptotically good codes correcting insertions, deletions, and
  transpositions.
\newblock {\em IEEE Transactions on Information Theory}, 45(7):2552--2557,
  1999.

\bibitem[TL08]{thrun2008simultaneous}
S.~Thrun and J.~J. Leonard.
\newblock Simultaneous localization and mapping.
\newblock In {\em Springer Handbook of Robotics}, pages 871--889. Springer,
  2008.

\bibitem[UHV15]{unl}
J.~Unnikrishnan, S.~Haghighatshoar, and M.~Vetterli.
\newblock Unlabeled sensing with random linear measurements.
\newblock {\em preprint arXiv:1512.00115}, 2015.

\bibitem[Wai15]{wainwrig2015high}
M.~J. Wainwright.
\newblock High-dimensional statistics: A non-asymptotic viewpoint.
\newblock {\em in preparation. University of California, Berkeley}, 2015.

\bibitem[YB99]{yang1999information}
Y.~Yang and A.~Barron.
\newblock Information-theoretic determination of minimax rates of convergence.
\newblock {\em Annals of Statistics}, pages 1564--1599, 1999.

\bibitem[Yos64]{yoshihara}
K.~Yoshihara.
\newblock Simple proofs for the strong converse theorems in some channels.
\newblock In {\em Kodai Mathematical Seminar Reports}, volume~16, pages
  213--222. Dept. of Mathematics, Tokyo Institute of Technology, 1964.

\bibitem[Yu97]{yu1997assouad}
B.~Yu.
\newblock Assouad, {F}ano, and {L}e {C}am.
\newblock In {\em Festschrift for Lucien Le Cam}, pages 423--435. Springer,
  1997.

\end{thebibliography}

\end{document}